\documentclass[a4paper,12pt]{article}

\usepackage[margin=3cm,footskip=1cm]{geometry}

\usepackage{color}

\usepackage[T1]{fontenc}
\usepackage{amsmath,amssymb,amsthm,bm,mathtools,xspace,booktabs,url}
\usepackage{graphicx,etoolbox}

\usepackage{mathabx}

\usepackage[shortlabels]{enumitem}
\setlist{
  listparindent=\parindent,
  parsep=0pt,
}

\usepackage[round]{natbib}

\usepackage[raggedright]{titlesec}

\usepackage[normalem]{ulem}

\numberwithin{equation}{section} 

\theoremstyle{plain} 
\newtheorem{theorem}{Theorem}[section]
\newtheorem{lemma}[theorem]{Lemma}
\newtheorem{proposition}[theorem]{Proposition}

\theoremstyle{definition} 

\usepackage{authblk}

\makeatletter
\newcommand\CorrespondingAuthor[1]{%
  \begingroup%
  \def\@makefnmark{}%
  \footnotetext{Corresponding author: #1}%
  \endgroup%
}
\makeatother

\makeatletter
\renewenvironment{abstract}{%
  \small%
  \providecommand\keywords{%
    \par\medskip\noindent\textit{Keywords:}\xspace}%
  \begin{center}%
    \bfseries \abstractname\vspace{-.5em}\vspace{\z@}%
  \end{center}%
  \quote%
}{\endquote}
\makeatother

\usepackage{siunitx}

\usepackage[all]{onlyamsmath}

\usepackage{lipsum}

\usepackage[utf8]{inputenc}
\usepackage{hyperref}
\usepackage{fixme}

\usepackage{subfigure}
\usepackage{float}

\DeclareMathOperator\Var{Var}
\DeclareMathOperator\Cov{Cov}
\DeclareMathOperator\EE{E}
\DeclareMathOperator\PP{P}
\DeclareMathOperator\tr{tr}

\newcommand{\ZZ}{\mathbb{Z}}
\newcommand{\RR}{\mathbb{R}}
\newcommand{\NN}{\mathbb{N}}

\newcommand{\bx}{\mathbf{x}}

\newcommand{\bX}{\mathbf{X}}

\newcommand{\bXu}{\mathbf{X}_{B_{u,n}}}

\newcommand{\lth}{\lambda_{\theta^\star}}

\newcommand{\dd}{\,\mathrm d}
\newcommand{\ii}{\mathbf i}

\newcommand{\LRL}{\mbox{{\sf LRL}}}
\newcommand{\LPL}{\mbox{{\sf LPL}}}
\newcommand{\Wna}{W_n\ominus \alpha_n}

\newcommand{\assPhi}{{\bf[$\mathbf {\Phi}$]}}
\newcommand{\assg}{{\bf[$\mathbf {g}$]}}
\newcommand{\assW}{{\bf[$\mathbf {W_n}$]}}

\begin{document}

\title{Parametric estimation of pairwise Gibbs point processes with infinite range interaction.}

\author[1]{Jean-Fran\c{c}ois Coeurjolly} 
\author[2,3]{Fr\'ed\'eric Lavancier}
\affil[1]{Univ. Grenoble Alpes, France\\
\texttt{Jean-Francois.Coeurjolly@upmf-grenoble.fr}.}
\affil[2]{Laboratoire de Math\'ematiques Jean Leray, Nantes University, France\\ 
 \texttt{Frederic.Lavancier@univ-nantes.fr}.}
 \affil[3]{ Inria, Centre Rennes Bretagne Atlantique, France}

\date{}


\maketitle

\begin{abstract}
  This paper is concerned with statistical inference for infinite range interaction Gibbs point processes and in particular for the large class of Ruelle superstable and lower regular pairwise interaction models. We extend classical statistical methodologies such as the pseudolikelihood and the logistic regression methods, originally defined and studied for finite range models. Then we prove that the associated  estimators  are strongly consistent and satisfy a central limit theorem, provided the pairwise interaction function  tends sufficiently fast to zero.  To this end, we introduce a  new central limit theorem for almost conditionally centered triangular arrays of random fields.

  \keywords Lennard-Jones potential; pseudolikelihood; central limit theorem.
\end{abstract}

\section{Introduction}\label{sec:intro}

Spatial Gibbs point processes are an important class of models used in spatial point pattern analysis \citep{vanLieshout:00,moeller:waagepetersen:04,illian:et:al:08}. Gibbs point processes can be viewed as modifications of the Poisson point process in order to introduce dependencies, such as attraction or repulsion, between points. These models initially arise from statistical  physics to approximate the interaction between pairs of particles \citep{ruelle:69,preston:76,georgii:88}. The most well-known example is the Lennard-Jones model \citep{jones1924} which yields repulsion at short scales and attraction at long scales.

Assuming   that the Gibbs model has a parametric form, an important question concerns the estimation of the parameters from a realization of the point process observed on a finite subset of $\RR^d$. 
 Popular solutions  include likelihood
\cite[e.g.][]{ogattane81,huang1999}, 
 pseudolikelihood
\cite[e.g.][]{besag:77,A-JenMol91,baddturn00} and  logistic regression \citep{baddeley:coeurjolly:rubak:waagepetersen:14}. 
The two latter methods are more interesting from a practical point of view as they avoid the computation of the normalizing constant in the likelihood, which is in most cases inaccessible for Gibbs point processes and must be approximated by simulation-based methods. We focus in this paper on the pseudolikelihood and logistic regression methods.

 When the Gibbs model is assumed to have a finite range interaction, consistency and asymptotic normality of  the pseudolikelihood and  logistic regression estimators are established in \cite{A-JenMol91,jensen:kunsch:94,billiot:coeurjolly:drouilhet:08,A-DerLav09,A-CoeDro10,baddeley:coeurjolly:rubak:waagepetersen:14}, for large families of Gibbs models. The finite range assumption means that there exists $R>0$ such that the particles do not interact at all 
 if they are at a distance greater than $R>0$ apart. For the two aforementioned  inference methods,  this assumption  turns out to be  crucial from  both a practical point of view and a theoretical point of view, as explained below. However this assumption may imply an artificial discontinuity of the interactions between particles, where two particles at a distance $R-\epsilon$ apart interact while they do not at a distance $R+\epsilon$, for any small $\epsilon>0$. This is for instance the case for the  widely used  Strauss model,  see e.g. \citet{moeller:waagepetersen:04}. In fact, this assumption rules out many interesting Gibbs models from statistical physics like the Lennard-Jones model. The purpose of this work is to extend the pseudolikelihood and logistic regression methods to infinite range interaction Gibbs models. 

From a practical point of view, an important issue is edge effects. Assume we observe a Gibbs point process with finite range interaction $R>0$ on a window $W\subset\RR^d$. Then the pseudolikelihood computed on $W$ actually depends on the point process on $W\oplus R$, where $W\oplus R$ denotes the dilation of $W$ by a ball with radius $R$. Some border correction is often used to make the pseudolikelihood score unbiased.   An obvious solution is to compute the pseudolikelihood on the eroded set $W\ominus R$, and in view of $(W\ominus R) \oplus R \subseteq W$ (see \citet{chiu2013}) the observation of the point process on $W$ is sufficient for the computation. 
From a theoretical point of view, standard technical tools for unbiased estimating equations are available to derive the asymptotic properties of the associated estimator.
 If the Gibbs point process has infinite range interaction, then the pseudolikelihood computed on $W$ depends on the point process over the whole space $\RR^d$. It is in general impossible to apply a border correction that preserves unbiasedness of the pseudolikelihood score function. 
We propose in Section~\ref{sec:method} a family of contrast functions that involve an eroded set, following the previous border correction, and a truncated range of interaction. 
The details are exposed in Section~\ref{sec:method}. However these contrast functions still lead to biased  score functions and the standard ingredients to derive consistency and asymptotic normality of the estimators do not apply.

The strong consistency of the maximum pseudolikelihood estimator was studied by \citet{mase:95} for pairwise interaction Gibbs point processes, including the infinite range interaction case, but under the assumption that the configuration of points outside $W$ is known.  Under the more realistic setting where the point process is observed only on $W$, we prove the strong consistency of our pseudolikelihood estimator in Proposition~\ref{prop:cvps}. 
Our result is valid for a large family of pairwise Gibbs models, namely the class of Ruelle superstable and lower regular models. The asymptotic normality is more challenging to establish. When the pseudolikelihood score function is unbiased, the main ingredient is a central limit theorem for conditionally centered random fields proved and generalized in \citet{guyon1992,jensen:kunsch:94,comets:janzura:98,dedecker:98,coeurjolly:lavancier:13}. 
It allows in particular to avoid mixing assumptions for Gibbs point process that are only known in restrictive frameworks (see for instance \citet{heinrich:92,jensen:93}). In our infinite range setting where the score function is biased, a new ingredient is needed. 
We establish in Appendix~\ref{sec:clt}  a new central limit theorem for triangular arrays of  {\it almost} conditionally centered random fields. 
This allows us to derive in Theorem~\ref{thm:cvloi} the asymptotic normality of our pseudolikelihood estimator. Assume the potential decreases with a hyperbolic rate with exponent $\gamma_2$. While $\gamma_2>d$ is (in general) required to ensure the existence of a Gibbs measure and the consistency of the pseudolikelihood estimator, we require the condition $\gamma_2>2d$ to ensure a central limit theorem. The asymptotic normality when $d<\gamma_2<2d$ remains an open question. Nonetheless our result includes the important example of the Lennard-Jones model in dimension $d=2$, for which $\gamma_2=6$. Proposition~\ref{prop:logistic} discusses similar asymptotic results for  the  logistic regression  estimator.

The remainder of this  paper is organized as follows. In Section~\ref{sec:method} we recall some basic facts about Gibbs point processes and we explain how to generalize inference methods for Gibbs models with infinite range interaction. 
We derive in Section~\ref{sec:IR} the asymptotic properties of our estimators. Section~\ref{sec:sim} contains a simulation study concerning the estimation of the Lennard-Jones potential, where some recommendations are derived for the practical choice of tuning parameters in the pseudolikelihood contrast function. 
Appendix~\ref{sec:clt}  contains our main technical tool, namely   a central limit theorem for  almost conditionally centered random fields, and Appendix~\ref{sec:auxiliary} gathers auxiliary lemmas.

\section{Background and statistical methodology}\label{sec:method}

\subsection{Notation}
We write $\Lambda\Subset\RR^d$ for a bounded set $\Lambda$ in $\RR^d$.
A configuration of points  $\bx$ is a locally finite subset of $\RR^d$, which means that the set $\bx_{\Lambda}:=\bx\cap \Lambda$ is finite for any set $\Lambda\Subset\RR^d$. We  use the notation $\bx_{\Lambda^c}=\bx\setminus\bx_{\Lambda}$ and  denote by $\Omega_0$ the space of all locally finite configurations  of points in $\RR^d$. For a $(p,q)$ matrix $M$ with real entries, we denote by $\|M\|=\tr(M^\top M)^{1/2}$ its Frobenius norm where $\tr$ is the trace operator and $M^\top$ is the transpose matrix of $M$. For a vector $z\in\RR^p$, $\|z\|$ reduces to its Euclidean norm. For a bounded set $E\subset \ZZ^d$, $|E|$ denotes  the number of elements of $E$, while  for $z\in \RR^p$ or $i\in\ZZ^p$, $|z|$ and $|i|$ stand for the uniform norm.

At many places in the document, we use the notation $c$ to denote a generic positive  constant which may vary from line to line.

\subsection{Pairwise interaction Gibbs point processes}\label{sec:gibbs}

We briefly recall the needed background material on point processes and we refer to \citet{daley:vere-jones:03} for more  details. A point process is a probability measure on $\Omega_0$.  The reference distribution on $\Omega_0$ is the homogeneous Poisson point process with intensity $\beta>0$, denoted by $\pi^\beta$. For $\Lambda\Subset\RR^d$, we  write $\pi^\beta_\Lambda$ for the restriction of $\pi^\beta$ to $\Lambda$. For any $\Delta\Subset\RR^d$ and $\bx\in \Omega_0$,  $N_\Delta(\bx)$ denotes the number of elements of $\bx\cap\Delta$.
Let $\Delta_i$ be the unit cube centered at $i\in \ZZ^d$.  We consider the following space of tempered configurations. 
\[ \Omega_T=\{\bx\in\Omega_0;\; \exists t>0, \forall n\geq 1, \sum_{i\in\ZZ^d, |i|\leq n} N^2_{ \Delta_i}(\bx) \leq t (2n+1)^d\}.\]
From the ergodic theorem (see \cite{guyon:95}), any second order stationary measure on $\Omega_0$ is supported on $\Omega_T$. We denote by $\Phi :\RR^d\to \RR\cup\{+\infty\}$ a pair potential function,  to which we associate the pairwise energy function $H_\Lambda : \Omega_T \to \RR\cup\{+\infty\}$, indexed by Borel sets $\Lambda\Subset\RR^d$ and defined by
\begin{equation}\label{defH}
H_\Lambda (\bx)=\frac 12 \sum_{\begin{subarray}{c}
  u,v\,\in\bx, \, u\neq v,\\ \{u,v\}\cap \bx_\Lambda\neq\emptyset \end{subarray}}  \Phi(u-v)\end{equation}
and we let
 \[\Omega = \{\bx\in\Omega_T,\,  \forall\Lambda\Subset\RR^d \; H_\Lambda(\bx)<\infty\}.\]

Following the Dobrushin-Lanford-Ruelle formalism, see \cite{preston:76}, we say that $P$ is a Gibbs measure with activity parameter $\beta>0$ for the pair potential function $\Phi$ if $P(\Omega)=1$ and for $P$-almost every configuration $\bx$ and any $\Lambda\Subset\RR^d$, the conditional law of $P$ given $\bx_{\Lambda^c}$
is absolutely continuous with respect to $\pi^\beta_\Lambda$  with the density  $\exp\{-H_\Lambda (\bx)\}/Z_\Lambda(\bx_{\Lambda^c})$, where $Z_\Lambda(\bx_{\Lambda^c})=\int_\Omega \exp\{-H_\Lambda (\bx_\Lambda \cup   \bx_{\Lambda^c})\} \, \pi_\Lambda^\beta(\dd\bx_{\Lambda^c})$ is the normalizing constant.

We use at many places in this paper the  GNZ equation, after \cite{A-Geo76} and \cite{nguyen:zessin:79b}, recalled below, which is a characterization of a Gibbs measure. It is given in terms  of the Papangelou conditional intensity $\lambda : \RR^d\times\Omega \to \RR_+$ defined for any $\Lambda\ni u$ by 
\begin{equation}\label{eq:pigpp}
\lambda(u,\bx) = \beta\, \frac{e^{-H_\Lambda(\bx \cup u)}}{e^{-H_\Lambda(\bx)}} =  \beta\, e^{-\sum_{v\in \bx}\Phi(v-u)}.  
\end{equation}
This quantity does not depend on $\Lambda$, provided $u\in\Lambda$, and  can be viewed as the conditional probability to have a point in a vicinity of $u$, given that the configuration elsewhere is $\bx$. 

\begin{theorem}[GNZ formula]
A probability measure $P$ on $\Omega$ is a Gibbs measure  with activity parameter  $\beta>0$  for the pair potential function $\Phi$ if for any measurable function $f:\Omega\times\RR^d \to \RR$ such that the following expectations are finite,
\begin{equation}\label{eq:GNZ}
\EE \left\{ \sum_{u\in\bX} f(u,\bX \setminus u)\right\}= \EE \left\{  \int  f(u,\bX) \lambda(u,\bX)  \dd u \right\}
\end{equation}
where $\EE$ denotes the expectation with respect to $P$. 
\end{theorem} 
This result can  be refined by  a  conditional version  stated in the following lemma. Its proof is actually part of the initial proof of~\eqref{eq:GNZ}, see also \citet[Proof of Theorem~2]{billiot:coeurjolly:drouilhet:08} for a particular case. We reproduce the demonstration below. 
\begin{lemma}[Conditional GNZ formula] Let  $P$ be a Gibbs measure  with activity parameter  $\beta>0$, with  pair potential  $\Phi$ and Papangelou conditional intensity $\lambda$. Then  for any measurable function $f:\Omega\times\RR^d \to \RR$ and for any  $\Lambda\Subset\RR^d$ such that the following expectations are finite 
\begin{equation}\label{condGNZ}
\EE  \left\{\sum_{u\in\bX_\Lambda} f(u,\bX\setminus u)   \mid \bX_{\Lambda^c}\right\} = \EE \left\{ \int_\Lambda  f(u,\bX)   \lambda(u,\bX)   \mid \bX_{\Lambda^c} \right\}
\end{equation}
where $\EE$ denotes the expectation with respect to $P$.
\end{lemma}

\begin{proof}
By definition of the conditional law of $P$ given $\bx_{\Lambda^c}^\prime$
\begin{align*}
\EE\left\{\sum_{u\in\bX_\Lambda} f(u,\bX\setminus u) \mid \bX_{\Lambda^c}  =  \bx_{\Lambda^c}' \right\} & = \int_{\Omega} \sum_{u\in\bx_\Lambda} f(u,\bx_\Lambda \cup   \bx_{\Lambda^c}' \setminus u) \frac{e^{-H_\Lambda(\bx_\Lambda \cup   \bx_{\Lambda^c}')}}{Z_\Lambda(\bx_{\Lambda^c}')} \pi_\Lambda^\beta(\dd\bx_\Lambda)\\
&= \EE_{\pi^\beta_\Lambda} \left\{\sum_{u\in\bX_\Lambda} f(u,\bX_\Lambda \cup   \bx_{\Lambda^c}' \setminus u) \frac{e^{-H_\Lambda(\bX_\Lambda \cup   \bx_{\Lambda^c}')}}{Z_\Lambda(\bx_{\Lambda^c}')} \right\}
\end{align*}
where $\EE_{\pi^\beta_\Lambda}$ denotes the expectation with respect to  $\pi_\Lambda^\beta$. From the Slivnyak-Mecke formula (\cite{slivnyak1962,mecke1968}), we know that for any admissible measurable function $h$ 
\[\EE_{\pi^\beta_\Lambda} \left\{\sum_{u\in\bX} h(u,\bX\setminus u) \right\}
= \beta \EE_{\pi^\beta_\Lambda} \left\{\int_{\RR^d} h(u,\bX) \dd u \right\}.\]
By definition of the Papangelou conditional intensity \eqref{eq:pigpp}, we also have for any $u\in\Lambda$,
$\beta\, e^{-H_\Lambda(\bx \cup u)} = e^{-H_\Lambda(\bx)}  \lambda(u,\bx)$.
Using these two facts, we conclude by
\begin{align*}
\EE  \bigg\{\sum_{u\in\bX_\Lambda} f(u,\bX\setminus u)  & \mid \bX_{\Lambda^c} = \bx_{\Lambda^c}' \bigg\}  \\
 &
 = \beta \EE_{\pi^\beta_\Lambda} \left\{
 \int_\Lambda  f(u,\bX_\Lambda \cup   \bx_{\Lambda^c}')   \frac{e^{-H_\Lambda(\bX_\Lambda \cup   \bx_{\Lambda^c}'\cup u)}}{Z_\Lambda(\bx_{\Lambda^c}')}\dd u \right\}\\
&= \EE_{\pi^\beta_\Lambda} \left\{\int_\Lambda  f(u,\bX_\Lambda \cup   \bx_{\Lambda^c}')   \lambda(u,\bX_\Lambda \cup   \bx_{\Lambda^c}') \frac{e^{-H_\Lambda(\bX_\Lambda \cup   \bx_{\Lambda^c}')}}{Z_\Lambda(\bx_{\Lambda^c}')}\dd u \right\}\\
&= \int_{\Omega}  \int_\Lambda  f(u,\bx_\Lambda \cup   \bx_{\Lambda^c}')   \lambda(u,\bx_\Lambda \cup   \bx_{\Lambda^c}') \frac{e^{-H_\Lambda(\bx_\Lambda \cup   \bx_{\Lambda^c}')}}{Z_\Lambda(\bx_{\Lambda^c}')} \dd u \, \pi_\Lambda^\beta(\dd\bx_\Lambda)\\
& = \EE \left\{ \int_\Lambda  f(u,\bX)   \lambda(u,\bX)   \mid \bX_{\Lambda^c}  =  \bx_{\Lambda^c}' \right\}. 
\end{align*}
\end{proof}

The existence of a Gibbs measure $P$ satisfying the above definition and characterization is a difficult question. Sufficient conditions on the pair potential $\Phi$ can be found in \cite{ruelle:69} and are also discussed in \cite{preston:76}. 
The special case of finite range potentials, i.e. compactly supported functions $\Phi$, is treated in  \cite{A-BerBilDro99}. As we are mainly interested in this paper by infinite range potentials, we introduce the following assumption, that leads to the existence of at least one stationary Gibbs measure, as proved in \cite{ruelle:69}.\\

\noindent \assPhi~The potential $\Phi$ is  bounded from below and there exist $0<r_1<r_2<\infty$, $c>0$ and $\gamma_1,\gamma_2>d$ such that $\Phi(u) \geq c \|u\|^{-\gamma_1}$ for $\|u\|\leq r_1$ and $|\Phi(u)| \leq c \|u\|^{-\gamma_2}$ for $\|u\|\geq r_2$.\\

Examples of potentials satisfying  \assPhi~are $\Phi(u)=\|u\|^{-\gamma}$ with $\gamma>d$ and $\Phi(u)=e^{-\|u\|}\|u\|^{-\gamma}$ with $\gamma>d$, in which cases  the assumption is satisfied with $\gamma_1=\gamma_2=\gamma$. Another important example is the general Lennard-Jones  pair potential defined for some $d<\gamma_2<\gamma_1$ and some $A,B>0$ by $\Phi(u)=A\|u\|^{-\gamma_1}-B\|u\|^{-\gamma_2}$.
The standard Lennard-Jones model corresponds to $d=2$, $\gamma_1=12$ and $\gamma_2=6$. The main interest of this model is that  it can model repulsion at small scales and attraction at large scales.

\subsection{Inference for infinite range Gibbs point processes} \label{sec:methodology}

In this section, we extend the usual statistical methodologies available for finite range Gibbs point processes to the infinite range case. We assume that the Gibbs measure  is parametric, in that the explicit expression of  the associated Papangelou conditional intensity \eqref{eq:pigpp} is entirely determined by the knowledge of  some parameter $\theta\in\Theta$, including the activity parameter $\beta>0$, where $\Theta$ is an open bounded set of $\RR^p$. We stress this assumption by writing  $\lambda_\theta$ instead of $\lambda$ and $\Phi_\theta$ instead of $\Phi$. For brevity, assumption \assPhi~now means that $\Phi_\theta$ fulfills this assumption for any $\theta\in \Theta$.

Assume that we observe the point process $\bX$ in $W_n$ where $(W_n)_{n\geq 1}$ is a sequence of bounded domains which converges to $\RR^d$ as $n\to \infty$. As outlined in the introduction, the pseudolikelihood and the logistic regression methods are popular alternatives to the maximum likelihood as they do not involve the normalizing constant. The associated estimators are respectively defined as the maximum of 
\begin{align}
\LPL_{W_n}(\bX;\theta) &= \sum_{u \in \bX_{W_n}} \log\lambda_\theta(u,\bX\setminus u)- \int_{W_n}\lambda_\theta(u,\bX)\dd u\label{lpl0}\\
\LRL_{W_n}(\bX;\theta) &= \sum_{u \in \bX_{W_n}} \log \frac{\lambda_\theta(u,\bX\setminus u)}{\lambda_\theta(u,\bX\setminus u)+\rho} - \int_{W_n}\rho \log \frac{\lambda_\theta(u,\bX)+\rho}{\rho}\dd u \label{lrl0}
\end{align}
where $\rho$ is some fixed positive real number. 

A problem however occurs. The integrals in \eqref{lpl0} and~\eqref{lrl0} are not computable in practice because for values of $u$ close to the boundary of $W_n$, $\lambda_\theta(u,\bX)$ depends on $\bX_{W_n^c}$ which is not observed. When $\bX$ has a finite range $0<R<\infty$, meaning that $\Phi_\theta$ is compactly supported on the euclidean ball $B(0,R)$ or equivalently that for any $u\in \RR^d$ and any $\bx \in \Omega$, $\lambda_\theta(u,\bx)=\lambda_\theta(u,\bx_{B(u,R)})$, we can simply substitute $W_n$ by $W_{n}\ominus R$  in \eqref{lpl0} and~\eqref{lrl0}, where for  $\Lambda\Subset \RR^d$ and some $\kappa\geq 0$ the notation $\Lambda\ominus \kappa$ stands for the domain $\Lambda$ eroded by the ball $B(0,\kappa)$. Using this border correction  $\lambda_\theta(u,\bX)$ can  be indeed computed for any $u\in W_{n}\ominus R$. As a remaining practical issue, the integrals have  to be approximated  by some numerical scheme or by Monte-Carlo, see \cite{baddeley:coeurjolly:rubak:waagepetersen:14} for an efficient solution. 
 
 The asymptotic properties of the pseudolikelihood and the logistic regression estimators are well understood in this finite range setting, see the references in introduction. Maximizing the log-pseudolikelihood (or the logistic regression likelihood) on $W_{n}\ominus R$  is equivalent to cancel the score, i.e. the gradient of $\LPL_{W_{n}\ominus R}(\bX;\theta)$ (or $\LRL_{W_{n}\ominus R}(\bX;\theta)$) with respect to $\theta$. The key-ingredient is that both scores constitute unbiased estimating functions, since by application of the GNZ formula~\eqref{eq:GNZ} their expectation vanishes when $\theta$ corresponds to  the true parameter of the underlying Gibbs measure. Standard theoretical tools for unbiased estimating equations  (see e.g.~\citet{guyon:95})  can therefore be used  to study the consistency and asymptotic normality of the associated estimators.\\

In the infinite range setting, the situation becomes more delicate since for any $u$, $\lambda_\theta(u,\bX)$  depends on $\bX_{\Lambda}$ for any $\Lambda\subset\RR^d$. In this case,  we introduce the following modifications of \eqref{lpl0} and~\eqref{lrl0} that depend on two sequences of positive real numbers $(\alpha_n)$ and $(R_n)$
\begin{align}
\widetilde\LPL_{\Wna,R_n}(\bX;\theta) &= \sum_{u \in \bX_{\Wna}} \log\lambda_\theta(u,\bX_{u,R_n}\setminus u)- \int_{\Wna}\lambda_\theta(u,\bX_{u,R_n})\dd u\label{lpl}\\
\widetilde\LRL_{\Wna,R_n}( \bX;\theta) &= \sum_{u \in \bX_{\Wna}} \log \frac{\lambda_\theta(u,\bX_{u,R_n}\setminus u)}{\lambda_\theta(u,\bX_{u,R_n}\setminus u)+\rho} - \int_{\Wna} \!\!\!\!\!\!\!\!\!\!\rho \log \frac{\lambda_\theta(u,\bX_{u,R_n})+\rho}{\rho}\dd u \label{lrl}
\end{align}
where we denote $\bX_{u,R_n}=\bX_{B(u,R_n)\cap W_n}$. These expressions are computable from the single observation of $\bX$ on $W_n$, provided the integrals are approximated as usual by numerical scheme or by Monte-Carlo. Since they depend on two sequences $\alpha_n$ and $R_n$, \eqref{lpl} and \eqref{lrl} actually form a general family of contrast functions, important particular cases being the choices
\begin{itemize}
\item $\alpha_n=R_n$, which agrees with the classical border correction for finite range interaction models with range $R$ taking $R_n=R$;
 
\item $R_n=\infty$, accounting for the maximal possible range of interaction;

\item  $R_n=\infty$ and $\alpha_n=0$, which is a particular case of the previous choice where in addition no erosion is considered. 
\end{itemize}
We study in the next section the asymptotic properties of estimators derived from \eqref{lpl} and \eqref{lrl} for a wide class of sequences $\alpha_n$ and $R_n$, and based on a simulation study in Section~\ref{sec:sim}, we give some recommendations for the choice of these sequences in practice.
From a theoretical point view, these contrast functions introduce new challenges since the gradients of $\widetilde\LPL_{\Wna,R_n}( \bX;\theta)$ and  $\widetilde\LRL_{\Wna,R_n}( \bX;\theta)$ are no  longer  unbiased estimating equations in the infinite range case. To overcome this difficulty we prove a new central limit theorem in Appendix~\ref{sec:clt} for almost conditionally centered random fields.

\section{Asymptotic properties} \label{sec:IR}

We present asymptotic properties of the maximum pseudolikelihood estimate, derived from~\eqref{lpl}, for infinite range Gibbs point process. Similar results  for the maximum logistic regression derived from~\eqref{lrl} are presented at the end of this section without proof. 
We assume the window of observation expands to $\RR^d$ as follows.\\

\noindent \assW~$(W_n)$ is an increasing sequence of convex compact sets, such that $W_n\to \RR^d$ as $n\to \infty$.

\medskip

We focus on exponential family models of pairwise interaction Gibbs point  processes and rewrite the model~\eqref{eq:pigpp} for any $u\in \RR^d$ and $\bx \in \Omega$ as 
\begin{equation} \label{eq:model}
\lambda_\theta(u,\bx)= \beta e^{-\sum_{v \in \bx} \Phi_\theta(v-u)} = e^{- \theta^\top t(u,\bx)}  
\end{equation}
with $\theta_1=- \log\beta$ and $t=(t_1,\dots,t_p)^\top$ where $t_1(u,\bx)=1$  and
\begin{equation}\label{deft}
t_m(u,\bx)=\sum_{v \in \bx} g_m(v-u),\quad m=2,\ldots,p.
\end{equation}
In that connection, our framework amounts to assume that $\Phi = \sum_{m=2}^p \theta_m g_m $.  For convenience we let $g_1=0$ and we denote by $g$ the $p$-dimensional vector   $g=(0,g_2,\ldots,g_p)^\top$.  We make the following assumption on $g$.\\

\noindent \assg~For all $m\geq 2$, $g_m$ is bounded from below and there exist $\gamma_1,\gamma_2>d$ and $c_g,r_0>0$ such that  
\begin{itemize}
\item[(i)] $\forall \|x\|<r_0$ and $\forall\theta\in \Theta$, $\theta_2 \, g_2(x)\geq c_g \|x\|^{-\gamma_1}$
\item[(ii)] $\forall m\geq 3$, $g_m(x)=o(\|x\|^{-\gamma_1})$ as $\|x\| \to 0$
\item[(iii)] $\forall m\geq 2$ and $\forall \|x\|\geq r_0$, $|g_m(x)| \leq c \|x\|^{-\gamma_2}.$
\end{itemize}

Since $\Theta$ is bounded, \assg~implies \assPhi~which yields that for any $\theta\in \Theta$ there exists  a Gibbs measure $P_\theta$. Assumption \assg \,  allows us  to specify which function $g_m$ is responsible for the behavior at the origin of $\Phi_\theta$, namely $g_2$.
Note that the Lennard-Jones model defined in Section~\ref{sec:gibbs} (and the other examples presented in this section) fits this setting with $\theta_2=A$, $\theta_3=-B$, $g_2(u)=\|u\|^{-\gamma_1}$ and $g_3(u)=\|u\|^{-\gamma_2}$. In the sequel, $\theta^\star$ stands for the true parameter vector to estimate. In other words, we assume observing  a realization of a spatial point process $\bX$ with Gibbs measure~$P_{\theta^\star}$ on $W_n$. 

For exponential family models~\eqref{eq:model} the score function of the log-pseudolikelihood defined by~\eqref{lpl} writes $s_{\Wna,R_n}(\bX;\theta)$ where for any $\Delta\subset \RR^d$
\begin{equation}\label{defs}
 s_{\Delta,R_n}(\bX;\theta) =   \int_{\Delta} t(u,\bX_{u,R_n}) \lambda_\theta(u,\bX_{u,R_n}) \dd u - \sum_{u\in \bX_{\Delta}} t(u,\bX_{u,R_n} \setminus u).
\end{equation}

Our first result establishes the strong consistency of the maximum pseudolikelihood based on~\eqref{lpl} for infinite range Gibbs point processes and for a wide class of sequences $(\alpha_n,R_n)$. In close relation, \citet{mase:95}  proved the strong consistency of estimators derived from~\eqref{lpl0}. As pointed out in Section~\ref{sec:methodology}, the form~\eqref{lpl0} of log-pseudolikelihood is however unusable as it can only be computed if $\bX$ is observed on $\RR^d$. We obtain the same result but for estimators derived from the computable pseudolikelihood given by~\eqref{lpl}.


\begin{proposition} \label{prop:cvps}
 Assume that \assW~and \assg~hold. 
 Then for any $\bx\in\Omega$ the function $\theta \to -\widetilde\LPL_{\Wna,R_n}(\bx;\theta)$ is a convex function  with Hessian matrix  given by
\begin{align}
  -\frac{\dd}{\dd\theta \dd \theta^\top}\,  \widetilde \LPL_{\Wna,R_n}(\bx;\theta)
  &=  - \frac{\mathrm d}{\mathrm d \theta^\top} s_{\Wna,R_n}(\bx;\theta) 
    \nonumber \\
&=\int_{\Wna} t(u,\bx_{u,R_n})t(u,\bx_{u,R_n})^\top \lambda_\theta(u,\bx_{u,R_n}) \dd u. \label{eq:derivScore}
\end{align}
Moreover if  $\alpha_n |W_n|^{-1/d}\to 0$ and $R_n\to\infty$ as $n\to\infty$, and if for any $y \in \RR^p\setminus \{0\}$ 
\begin{equation}\label{eq:ident}
  P\left\{ y^\top t(0,\bX) \neq 0\right\} >0,
\end{equation}
then the maximum pseudolikelihood estimator 
\[
\widehat\theta_{\scriptsize\widetilde{\LPL}} = \mathrm{argmax}_{\theta\in \Theta} \widetilde{\LPL}_{\Wna,R_n}(\bX;\theta)  
\] 
converges almost surely to $\theta^\star$ as $n\to \infty$.
\end{proposition}

\begin{proof}
By \assW~and the basic assumption on $\alpha_n$, we are ensured that $\Wna$ is a sequence of regular bounded domains of $\RR^d$ and that $|\Wna|\to \infty$ as $n\to \infty$.
Since any stationary Gibbs measure can be represented as a mixture of ergodic measures \citep{preston:76}, it is sufficient to prove consistency for ergodic measures. So, we assume here that $P_{\theta^\star}$ is ergodic. Since $\Theta$ is an open bounded set, and by convexity of $\theta \to -\widetilde\LPL_{\Wna,R_n}(\bx;\theta)$, then from \citet[Theorem 3.4.4]{guyon:95} we only need to prove that $K_n(\theta,\theta^\star)= |\Wna|^{-1} \left\{ \widetilde\LPL_{\Wna,R_n}(\bX;\theta^\star) - \widetilde \LPL_{\Wna,R_n}( \bX;\theta) \right\} \to K(\theta,\theta^\star)$ almost surely as $n\to \infty$, where $\theta\to K(\theta,\theta^\star)$ is a nonnegative function which vanishes at $\theta=\theta^\star$ only. We decompose $K_n(\theta,\theta^\star)$ as the sum of the three terms $T_{1}+T_2(\theta^\star)-T_2(\theta)$ where for any $\theta\in\Theta$
\begin{align*}
T_1&= |\Wna|^{-1} \left\{ \LPL_{\Wna}(\bX;\theta^\star) - \LPL_{\Wna}( \bX;\theta)\right\} \\
T_2(\theta)&= |\Wna|^{-1} \left\{ \widetilde\LPL_{\Wna,R_n}( \bX;\theta) - \LPL_{\Wna}( \bX;\theta)\right\}. 
\end{align*}
Lemma~\ref{lem:int} shows in particular that $\lambda_\theta(0, \bX)$ and $|\theta^\top t(0, \bX)|\lth(0, \bX)$ have finite expectation under  $P_{\theta^\star}$. Hence, using the ergodic theorem for spatial processes of \cite{nguyen:zessin:79}, we can follow the proof of \citet{mase:95} or the proof of \citet[Theorem~1]{billiot:coeurjolly:drouilhet:08} to prove that $T_1\to K(\theta,\theta^\star)$ almost surely as $n\to \infty$ where
\[ 
  K(\theta,\theta^\star) = \EE \left(
\lth(0,\bX) \left[
  e^{(\theta^\star-\theta)^\top t(0,\bX)} - \{1 + (\theta^\star-\theta)^\top t(0,\bX)\}
  \right]
  \right)
\]
which is a nonnegative function that vanishes at $\theta=\theta^\star$ only, under the identifiability condition~\eqref{eq:ident}.
So the rest of the proof consists in proving that $T_2(\theta)\to 0$ almost surely for any $\theta\in\Theta$. We have $T_2(\theta)=T_1^\prime + T_2^\prime$ where
\begin{align*}
  T_1^\prime & = |\Wna|^{-1}\sum_{u\in  \bX_{\Wna}} \theta^\top \left\{ t(u,\bX_{u,R_n}\setminus u)-t(u,\bX\setminus u)\right\} \\
  T_2^\prime & = |\Wna|^{-1}\int_{\Wna} \left\{ \lambda_\theta(u, \bX) -\lambda_\theta(u, \bX_{u,R_n})\right\}\dd u. \\
\end{align*}
Let us first look at  $T_1^\prime$. By boundedness of $\Theta$, it is sufficient to prove that  $|\Wna|^{-1}\sum_{u\in  \bX_{\Wna}}  \left\{t_m(u,\bX_{u,R_n}\setminus u)-t_m(u,\bX\setminus u)\right\}$ tends to 0  almost surely. 
Let $\delta_n$ be a sequence of real numbers such that $\delta_n\to \infty$ and $\delta_n |W_n|^{-1/d}\to 0$ as $n\to\infty$. For brevity, let $\bX_{(u,R_n)^c}:=\bX_{\RR^d\setminus \{W_n\cap B(u,R_n)\}}$.
\begin{align}\label{T1as}
&\sum_{u\in  \bX_{\Wna}}  \left\{t_m(u,\bX_{u,R_n}\setminus u)-t_m(u,\bX\setminus u)\right\} \nonumber\\
&= \sum_{u\in  \bX_{W_n\ominus \alpha_n}} \sum_{v\in \bX_{(u,R_n)^c}} g_m(v-u) \nonumber\\
&= \sum_{u\in  \bX_{W_n\ominus (\alpha_n+\delta_n)}} \sum_{v\in \bX_{(u,R_n)^c}} g_m(v-u) +  \sum_{u\in  \bX_{W_n\ominus \alpha_n\setminus W_n\ominus (\alpha_n+\delta_n)}} \sum_{v\in \bX_{(u,R_n)^c}} g_m(v-u)
\end{align}   
In the first sum above, $\|u-v\|\geq R_n\wedge (\alpha_n+\delta_n)$ and using the same arguments and the same notation as in (ii) of Lemma~\ref{lem:XWn} we get that the absolute value of this sum is lower than $c\, {(R_n\wedge (\alpha_n+\delta_n))^{-\gamma^\prime}} \sum_{u\in  \bX_{W_n\ominus (\alpha_n+\delta_n)}}H(u, \bX\setminus u)$ for some $\gamma'>0$. Hence
\begin{align*}
&|\Wna|^{-1} \left| \sum_{u\in  \bX_{W_n\ominus (\alpha_n+\delta_n)}} \sum_{v\in \bX_{(u,R_n)^c}} g_m(v-u) \right| \\
&\leq c\, {(R_n\wedge (\alpha_n+\delta_n))^{-\gamma^\prime}} \frac{|W_n\ominus (\alpha_n+\delta_n)|}{|\Wna|}  |W_n\ominus (\alpha_n+\delta_n)|^{-1}\!\!\! \sum_{u\in  \bX_{W_n\ominus (\alpha_n+\delta_n)}}H(u, \bX\setminus u).
\end{align*}   
By Lemma~\ref{lem:int}, the random variable $|H(0,\bX)| \lth(0,\bX)$ has finite expectation under $P_{\theta^\star}$. Moreover our assumptions ensure that $W_n\ominus (\alpha_n+\delta_n)$ is a sequence of regular bounded domains of $\RR^d$ with $|W_n\ominus (\alpha_n+\delta_n)|\to \infty$ as $n\to \infty$. So by the ergodic theorem  $ |W_n\ominus (\alpha_n+\delta_n)|^{-1}  \sum_{u\in  \bX_{W_n\ominus (\alpha_n+\delta_n)}} H(u,\bX\setminus u) \to \EE\left\{ H(0,\bX)\lth(0,\bX)\right\}$ almost surely whereby 
\begin{equation}\label{cvinter1}
|\Wna|^{-1} \left| \sum_{u\in  \bX_{W_n\ominus (\alpha_n+\delta_n)}} \sum_{v\in \bX_{(u,R_n)^c}} g_m(v-u) \right|  \to 0
\end{equation}
almost surely. For the second sum in \eqref{T1as}, using the notation  $|t_m|(u,\bx) = \sum_{v \in \bx} |g_m(v-u)|$ we have 
\begin{align*}
&|\Wna|^{-1}  \left|\sum_{u\in  \bX_{W_n\ominus \alpha_n\setminus W_n\ominus (\alpha_n+\delta_n)}} \sum_{v\in X_{(u,R_n)^c}} g_m(v-u) \right| \\
& \leq |\Wna|^{-1}  \sum_{u\in  \bX_{W_n\ominus \alpha_n\setminus W_n\ominus (\alpha_n+\delta_n)}} |t_m|(u,\bX)  \\
&=   |\Wna|^{-1}  \sum_{u\in  \bX_{W_n\ominus \alpha_n}}  |t_m|(u,\bX)  -  |\Wna|^{-1}  \sum_{u\in \bX_{W_n\ominus (\alpha_n+\delta_n)}} |t_m|(u,\bX) 
\end{align*}
which tends to 0 almost surely by application of the ergodic theorem, Lemma~\ref{lem:int} and since $|W_n\ominus(\alpha+\delta_n)|/|W_n\ominus \alpha_n| \to 1$ as $n\to \infty$.
Combining this result with \eqref{cvinter1} in \eqref{T1as} shows that $T_1^\prime\to 0$ almost surely.

Concerning the  remaining term $T_2^\prime$,  we have 
\begin{equation*}
T_2^\prime = |\Wna|^{-1} \int_{\Wna} \lambda_\theta(u, \bX) \left\{ 1 - e^{-\theta^\top  \sum_{v\in \bX_{(u,R_n)^c}} g(v-u)} \right\} \dd u.
\end{equation*}
We can use exactly the same decomposition as in \eqref{T1as} by introducing $\delta_n$, then use Lemmas~\ref{lem:XWn} and \ref{lem:int} to apply the ergodic theorem, leading to $T_2^\prime\to 0$ almost surely. These details are omitted. Hence $T_2(\theta)\to 0$ almost surely for any $\theta \in \Theta$ and the proof of Proposition~\ref{prop:cvps}  is completed.
\end{proof}


As a preliminary result towards  the asymptotic normality of $\widehat\theta_{\scriptsize\widetilde{\LPL}}$, we state in the next lemma general conditions on the sequences $\alpha_n$ and $R_n$ leading to the equivalence in probability of the score functions, up to $|W_n|^{1/2}$. 
\begin{lemma}\label{lemma_score}
Assume that  \assW~and \assg~hold. Let $\mathcal D$ be the set of all sequences $(\alpha_n,R_n)$ such that $\alpha_n |W_n|^{-1/d}\to 0$ and such that there exists $0<\gamma'<\gamma_2-d$ such that $\alpha_n^{-\gamma'} |W_n|^{1/2}\to 0$ and $R_n^{-\gamma'} |W_n|^{1/2}\to 0$.
Then, for any $(\alpha_n,R_n)\in\mathcal D$ and $(\alpha'_n,R'_n)\in\mathcal D$ we have
\begin{equation}\label{equivDn}
s_{\Wna,R_n}(\bX;\theta^*) - s_{W_n\ominus\alpha_n',R'_n}(\bX;\theta^*)  = o_P(|W_n|^{1/2}).
\end{equation}
\end{lemma}

\begin{proof}
For $\Delta\Subset\RR^d$, we denote for short $s_{\Delta,R_n}:=s_{\Delta,R_n}(\bX;\theta^*)$ and set
\begin{equation}\label{sprime}
s'_{\Delta,R_n} =  \int_{\Delta} t(u,\bX_{u,R_n}) \lth(u,\bX) \dd u - \sum_{u\in \bX_{\Delta}} t(u,\bX_{u,R_n} \setminus u).\end{equation}
We prove below that for any $(\alpha_n,R_n)\in\mathcal D$, 
$s_{\Wna,R_n} - s'_{W_n,\infty} = o_P(|W_n|^{1/2})$, whereby \eqref{equivDn} is an immediate consequence. We have
\begin{equation}
s_{\Wna,R_n} - s'_{W_n,\infty} = A + B + C
\end{equation}
with $A=s_{\Wna,R_n} - s_{\Wna,\infty}$, $B=s_{\Wna,\infty} -s'_{\Wna,\infty}$ and $C=s'_{\Wna,\infty}-s'_{W_n,\infty}$.
Let us prove that each of these three terms is $o_P(|W_n|^{1/2})$.

First, we have $A= A_1+A_2$ with
\begin{align*}
A_1&=\int_{\Wna} \{t(u,\bX_{u,R_n}) \lth(u,\bX_{u,R_n}) -   t(u,\bX_{W_n}) \lth(u,\bX_{W_n})\}\dd u,\\
A_2 &= \sum_{u\in \bX_{\Wna}} \{t(u,\bX_{W_n} \setminus u)-t(u,\bX_{u,R_n} \setminus u)\}.
\end{align*}
Both terms above are handled similarly and we give the details for $A_2$ only. Denoting $A_{2,m}$
the $m$-th coordinate of $A_2$, we obtain using the same arguments and the same notation as in Lemma~\ref{lem:XWn} (ii) that for any $m$
\begin{align*}
|A_{2,m}|\leq \sum_{u\in \bX_{\Wna}} \sum_{v\in \bX_{W_n\setminus B(u,R_n)}} |g_m(v-u)| \leq R_n^{-\gamma'} \sum_{u\in \bX_{W_n}}  H(u,\bX).
\end{align*}
Applying the GNZ formula and Lemma~\ref{lem:int}, we deduce that $\EE|A_{2,m}| = \mathcal O(R_n^{-\gamma'} |W_n|)$ showing that $|W_n|^{-1/2} A_2=o_P(1)$. The same result holds for $A_1$ by the arguments developed in Lemma~\ref{lem:XWn}~(ii)-(iii).  

Second, 
\begin{equation*} B=\int_{\Wna} t(u,\bX_{W_n}) \{ \lth(u,\bX_{W_n}) - \lth(u,\bX) \} \dd u.\end{equation*}
Note that 
\begin{align*}
|\lambda_\theta(u,\bx)-\lambda_\theta(u,\bx_{W_n})| &= \lambda_\theta(u,\bx) \Big| 1-  e^{\sum_{v\in \bx_{W_n^c}}\Phi_\theta(v-u)} \Big|
\end{align*}
where in the last sum $\|v-u\|\geq \alpha_n$ since $u\in\Wna$. The same arguments as in  the proof of Lemma~\ref{lem:XWn}~(iii) thus apply and we get for all $u\in\Wna$, 
$|\lambda_\theta(u,\bx)-\lambda_\theta(u,\bx_{W_n})| \leq c\,  \alpha_n^{-\gamma'} e^{c\,G(u,\bx)} \lambda_\theta(u,\bx)  H(u,\bx)$. From Lemma~\ref{lem:XWn}~(i) and  Lemma~\ref{lem:int}, we obtain
\begin{equation*} 
\EE |B|=  \mathcal O(\alpha_n^{-\gamma'} |\Wna|) =  \mathcal O(\alpha_n^{-\gamma'} |W_n|)
\end{equation*} 
 and thus $|W_n|^{-1/2} B=o_P(1)$.

Third, 
\begin{align*}
C&=-\int_{W_n\setminus(\Wna)}t(u,\bX_{W_n}) \lth(u,\bX) \dd u  +\sum_{u\in \bX_{W_n\setminus(\Wna)}} t(u,\bX_{W_n} \setminus u)\\
&=-\sum_{j\in\mathcal J_n} s'_{\Delta_{n,j},\infty}
\end{align*}
where $\Delta_j$ is the unit cube centered at $j\in \ZZ^d$, $\Delta_{n,j}=\Delta_j \cap (W_n\setminus(\Wna))$ and $\mathcal J_n\subset \ZZ^d$ is the set such that $W_n\setminus(\Wna)=\cup_{j\in \mathcal J_n} \Delta_{n,j}$.
We have 
\begin{align*}
\Var(C)&=\sum_{j,k\in\mathcal J_n}\Cov(s'_{\Delta_{n,j},\infty} , s'_{\Delta_{n,k},\infty}).
\end{align*}
It is not difficult to check that all results of Lemma~\ref{lem:cov}~(i)-(iii) remain true if the ball $B_{u,n}=B(u,\alpha_n)$ therein is replaced by $W_n$, or in other words $s'_\Delta$  is replaced by $s'_{\Delta,\infty}$. Therefore, from~(iii) of Lemma~\ref{lem:cov}
\[ \|  \Var(C)\| = \mathcal O(|\mathcal J_n|)=\mathcal O(|W_n\setminus(\Wna)|)=\mathcal O(|W_n|^{1-1/d}\alpha_n)\]
showing that $|W_n|^{-1} \|\Var(C)\| \to 0$. Hereby $|W_n|^{-1/2} C=o_P(1)$ and the proof is completed.
\end{proof}

The next result establishes the asymptotic normality of the score function associated to  the modified pseudolikelihood 
 $\widetilde\LPL_{\Wna,R_n}(\bX;\theta)$ at the true value of the parameter $\theta=\theta^\star$, whenever $(\alpha_n,R_n)$ belongs to the set $\mathcal D$ introduced in Lemma~\ref{lemma_score}.  The proof relies on a new central limit theorem stated in Appendix~\ref{sec:clt}. As a consequence we deduce the asymptotic normality of $\widehat\theta_{\scriptsize\widetilde{\LPL}}$.

These results require the following notation: let $\Sigma_\infty$ and $U_\infty$ the $(p,p)$ matrices 
\begin{align}
\Sigma_\infty =& \EE  \big\{ t(0,\bX)t(0,\bX)^\top \lth (0,\bX)\big\} \nonumber\\
&+ \int_{\RR^d} \EE  \big\{ t(0,\bX)t(v,\bX)^\top \lth (0,\bX)
\lth (v,\bX)  \big\}  \{ 1-e^{-\Phi_{\theta^\star}(v)} \}\dd v\nonumber\\
&+ \int_{\RR^d} \EE  \big\{  \lth (0,\bX)
\lth (v,\bX)   \big\}  g(v)g(v)^\top e^{-\Phi_{\theta^\star}(v)}\dd v \label{eq:SigmaInfty}\\
U_\infty =& \EE  \big\{ t(0,\bX)t(0,\bX)^\top \lth (0,\bX)\big\}. \label{eq:UInfty}
\end{align}
These matrices   are indeed correctly defined, as \assg~implies on the one hand that all the expectations involved are uniformly bounded in $v$ by Lemmas~\ref{lem:XWn}-\ref{lem:int}, and on the other hand that  
\[
  \int_{\RR^d} |1-e^{-\Phi_{\theta^\star}(v)}|\dd v<\infty 
  \quad \mbox{ and } \quad
  \int_{\RR^d} \|g(v)g(v)^\top\| e^{-\Phi_{\theta^\star}(v)} \dd v<\infty.
\]
We denote by $\overset{d}{\to}$  the  convergence in distribution.

\begin{theorem}\label{thm:cvloi}
Under the assumptions of Lemma~\ref{lemma_score} with $\gamma_2>2d$ and the assumption that $\Sigma_\infty$ is a positive definite matrix, then  we have the two following convergences in distribution for any $(\alpha_n,R_n)\in\mathcal D$, as $n\to \infty$,\\
(i)
\[ |W_n|^{-1/2} s_{\Wna,R_n}(\bX;\theta^\star) \overset{d}{\to} \mathcal N(0,\Sigma_\infty),\] 
(ii)
\[
  |W_n|^{1/2} \left(\widehat\theta_{\scriptsize\widetilde{\LPL}}-\theta^\star \right) \stackrel{d}{\to} \mathcal N \left( 0, U_{\infty}^{-1} \Sigma_\infty U_\infty^{-1} \right).
\]
\end{theorem}

Some remarks on this theorem are in order. 
The condition  $\gamma_2>2d$ is clearly the most restrictive one. Nonetheless it includes the standard Lennard-Jones model in dimension $d=2$ for which $\gamma_2=6$. Under \assg, existence of the model is ensured if $\gamma_2>d$ but it remains an open problem to prove the asymptotic normality of the pseudolikelihood estimator when $d<\gamma_2<2d$. 
Concerning the set $\mathcal D$ of possible sequences $(\alpha_n,R_n)$, it includes the natural choices $(\alpha_n,R_n)=(\alpha_n,\alpha_n)$ and $(\alpha_n,R_n)=(\alpha_n,\infty)$ discussed in Section~\ref{sec:methodology}, provided $\alpha_n$ tends to infinity at a good rate. However,  $\mathcal D$ does not include the particular case $(\alpha_n,R_n)=(0,\infty)$, whereas this choice leads to a consistent estimator as proved in Proposition~\ref{prop:cvps}. In fact, when the erosion parameter $\alpha_n$ does not tend to infinity, some edge effects occur due to the infinite range of the process. These edge effects are negligible with respect to  $|W_n|$ but not with respect to $ |W_n|^{1/2}$. 
Finally, following  \cite{coeurjolly:rubak:13}, it is possible to construct a fast estimator of the asymptotic covariance matrices $\Sigma_\infty$ and $U_\infty$, but its asymptotic properties are out of the scope of the present paper.

\begin{proof}
(i) Since $\gamma_2>2d$, there exists $\varepsilon>0$ and $a>0$ such that 
\begin{equation}
  \label{eq:a}
  \frac{d}{2(\gamma_2-d-\varepsilon)}<a<\frac1{2+\varepsilon}.
\end{equation}
For such $\varepsilon,a$, we let $\gamma^\prime=\gamma_2-d-\varepsilon$ and consider the particular case $R_n=\alpha_n= |W_n|^{a/d}$. Then, $\alpha_n |W_n|^{-1/d}\to 0$ and $\alpha_n^{-\gamma'} |W_n|^{1/2} = R_n^{-\gamma'} |W_n|^{1/2} = |W_n|^{\frac{d/2-a(\gamma_2-d-\varepsilon)}d}\to 0$.
 This particular choice of $(\alpha_n, R_n)$ thus belongs to $\mathcal D$. From Lemma~\ref{lemma_score}, if we prove~(i) for this choice, then the same convergence holds true for all sequences $(\alpha_n,R_n)$ in $\mathcal D$,  completing the proof of~(i). Henceforth we let $R_n=\alpha_n=|W_n|^{a/d}$ where $a$ is such that~\eqref{eq:a} holds. 

Denoting by $\Delta_j$ the unit cube centered at $j\in \ZZ^d$, we let $\Delta_{n,j}=\Delta_j \cap (\Wna)$ and $\mathcal I_n\subset \ZZ^d$ the set such that $\Wna=\cup_{j\in \mathcal I_n} \Delta_{n,j}$.  At several places in the proof the sequence $\rho_n=|\mathcal I_n|^{1/2}/\alpha_n^{\gamma^\prime}$ is involved. Then,
\begin{equation}\label{eq:rhon}
  \rho_n =\frac{|\mathcal I_n|^{1/2}}{\alpha_n^{\gamma^\prime}} = \mathcal O \left\{|W_n|^{\frac{d/2-a(\gamma_2-d-\varepsilon)}d} \right\}  
\end{equation}
tends to 0.

We write for short $s_{\Delta_{n,j}}=s_{\Delta_{n,j},R_n}(\bX;\theta^\star)=s_{\Delta_{n,j},\alpha_n}(\bX;\theta^\star)$ and $B_{u,n}=B(u,\alpha_n)$. Note that for any $u\in\Wna$,  $t(u,\bX_{u,R_n})=t(u,\bX_{u,\alpha_n})=  t(u,\bX_{B_{u,n}})$ and similarly $\lth(u,\bX_{u,R_n})=\lth(u,\bX_{u,\alpha_n})=\lth(u,\bX_{B_{u,n}})$. Therefore for any $j\in\mathcal I_n$, 
\[s_{\Delta_{n,j}}=\int_{\Delta_{n,j}} t(u,\bX_{B_{u,n}}) \lambda_\theta(u,\bX_{B_{u,n}}) \dd u - \sum_{u\in \bX_{\Delta_{n,j}}} t(u,\bXu \setminus u).\]
Letting $Z_{n,j}=s_{\Delta_{n,j}} - \EE (s_{\Delta_{n,j}})$, we have 
\[
  s_{\Wna}= S_n+ \EE(s_{\Wna})
\]
where $S_n= \sum_{j \in \mathcal I_n} Z_{n,j}$. Define 
\[
  \widehat\Sigma_n = \sum_{j \in \mathcal I_n} \sum_{
  \begin{subarray}{c}
  k \in \mathcal I_n \\ |k-j|\leq \alpha_n
  \end{subarray}} Z_{n,j} Z_{n,k}^\top \qquad \mbox{ and } \qquad \Sigma_n = \EE \widehat\Sigma_n. 
\]
The proof of (i) is completed if we show that $\Sigma_n^{-1/2} S_n \overset{d}{\to} \mathcal N(0,I_p)$,  $\Sigma_n^{-1/2} \EE(s_{\Wna})\to 0$ and $|W_n|^{-1} \Sigma_n\to \Sigma_\infty$. Let us prove the first convergence by application of Theorem~\ref{tcl1}. 

By \assW~and the definition of $\mathcal I_n$, we have $|\mathcal I_n|=\mathcal O(|W_n|)$, see e.g.  \citet[Lemma A.1]{coeurjolly:moeller:14}. 
From~\eqref{eq:a}, $\alpha_n^{(2+\varepsilon)d}=o(|\mathcal I_n|)$  which, following the remark after Theorem~\ref{tcl1}, satisfies the assumption of Theorem~\ref{tcl1} if Assumption~(a) of this theorem is satisfied for any $q\geq 1$. And the latter holds by definition of $Z_{n,j}$ and Lemma~\ref{moments}.

 Concerning assumption~(b), we use for short the notation $s'_{\Delta} = s'_{\Delta,\alpha_n}$ for any $\Delta\Subset\RR^d$, see~\eqref{sprime}, namely
\[s'_{\Delta} =  \int_{\Delta} t(u,\bX_{B_{u,n}}) \lth(u,\bX) \dd u - \sum_{u\in \bX_{\Delta}} t(u,\bX_{B_{u,n}} \setminus u).\]\
 Note that from the GNZ formula $\EE s'_{\Delta} =0$. We have from Lemma~\ref{lem:cov}, for any sequence $\mathcal J_n\subset \mathcal I_n$ such that $|\mathcal J_n|\to \infty$,
\begin{align*}
& \sum_{
 \begin{subarray}{c}
 j,k\in \mathcal J_n  
 \end{subarray}
 } \big\| 
\EE ( Z_{n,j}Z_{n,k}^\top)
  \big\|  = \sum_{
 \begin{subarray}{c}
 j,k\in \mathcal J_n  
 \end{subarray}
 } \big\| 
\Cov( s_{\Delta_{n,j}} ,s_{\Delta_{n,k}})
  \big\| \\ 
  & \leq    \sum_{
 \begin{subarray}{c}
 j,k\in \mathcal J_n
 \end{subarray}
 } \big\| 
\Cov( s_{\Delta_{n,j}} ,s_{\Delta_{n,k}}) -\Cov( s'_{\Delta_{n,j}} ,s'_{\Delta_{n,k}})
  \big\|  + 
  \sum_{
 \begin{subarray}{c}
 j,k\in \mathcal J_n
 \end{subarray}
 } \big\| 
\Cov( s'_{\Delta_{n,j}} ,s'_{\Delta_{n,k}})
  \big\|\\
  &  \leq    \sum_{
 \begin{subarray}{c}
 j,k\in \mathcal J_n
 \end{subarray}
 }  \left( \frac{c}{\alpha_n^{\gamma^\prime} (1+|k-j|^{\gamma_2})}  +
\frac{c}{\alpha_n^{2\gamma^\prime}}\right)
 + \sum_{
 \begin{subarray}{c}
 j,k\in \mathcal J_n  \\
 |j-k|\leq 2r_0
 \end{subarray}
 } \!\!\!\big\| 
\Cov( s'_{\Delta_{n,j}} ,s'_{\Delta_{n,k}})
  \big\| + \sum_{
 \begin{subarray}{c}
 j,k\in \mathcal J_n  \\
 |j-k|> 2r_0
 \end{subarray}
 } \!\!\!\!c |k-j|^{-\gamma_2} \\
 & \leq c\, \alpha_n^{-\gamma^\prime} |\mathcal J_n| +
c |\mathcal J_n| \rho_n^2
 + c   |\mathcal J_n| \|\Var( s'_{\Delta_{n,j}})\|  + c  |\mathcal J_n|
 \end{align*}
which is $\mathcal O( |\mathcal J_n| )$ by Lemma~\ref{moments} and from~\eqref{eq:rhon}. 

Since $\Sigma_\infty$ is  assumed to be a positive definite matrix, assumption~(c)  holds if we prove  that $|\mathcal I_n|^{-1} \Sigma_n \to \Sigma_\infty$ as $n\to \infty$. For this, let 
\[
\Sigma_n'=\Var(s'_{\Wna}) =  \sum_{ 
  \begin{subarray}{c}
j,k \in \mathcal I_n
  \end{subarray}} \EE \left\{ s'_{\Delta_{n,j}} {(s'_{\Delta_{n,k}})}^\top \right\}
\quad \mbox{ and } \quad 
   \widehat\Sigma_n'  = \sum_{
  \begin{subarray}{c}
  j,k \in \mathcal I_n \\ |k-j|\leq \alpha_n
  \end{subarray}} s'_{\Delta_{n,j}} {(s'_{\Delta_{n,k}})}^\top.  
\]
We have $ \|\; |\mathcal I_n|^{-1} \Sigma_n -\Sigma_\infty\| \leq T_1 + T_2 + T_3 $ 
where 
\begin{align*}
T_1  &= |\mathcal I_n|^{-1} \|\EE(\widehat \Sigma_n) - \EE(\widehat \Sigma_n') \|, \quad
T_2  = |\mathcal I_n|^{-1} \| \EE(\widehat \Sigma_n')-\Sigma_n' \|,\quad
T_3  = \| |\mathcal I_n|^{-1} \Sigma_n'  - \Sigma_\infty\|.
\end{align*}
First  applying Lemma~\ref{lem:cov}
\begin{align*}
T_1 & \leq   |\mathcal I_n|^{-1}  \sum_{j \in \mathcal I_n} \sum_{
  \begin{subarray}{c}
  k \in \mathcal I_n \\ |k-j|\leq \alpha_n    \end{subarray}} \| \EE (Z_{n,j} Z_{n,k}^\top) - \EE\{s'_{\Delta_{n,j}} {(s'_{\Delta_{n,k}})}^\top\} \| \\
  &  =  |\mathcal I_n|^{-1}   \sum_{j \in \mathcal I_n} \sum_{
  \begin{subarray}{c}
  k \in \mathcal I_n \\ |k-j|\leq \alpha_n    \end{subarray}} \big\| 
\Cov( s_{\Delta_{n,j}} ,s_{\Delta_{n,k}}) -\Cov( s'_{\Delta_{n,j}} ,s'_{\Delta_{n,k}})
  \big\| \\
  & = \mathcal O(\alpha_n^{-\gamma^\prime}) + \mathcal O(\alpha_n^{d-2\gamma^\prime})
  \end{align*}
and $T_1\to 0$ since $\gamma_2>2d$ implies $2\gamma^\prime-d>0$ as soon as $\varepsilon<d/2$, which can be assumed without loss of generality. Second, 
from (iii) in Lemma~\ref{lem:cov}
\begin{align*}
 T_2 & \leq  |\mathcal I_n|^{-1}  \sum_{j \in \mathcal I_n} \sum_{
  \begin{subarray}{c}
  k \in \mathcal I_n \\ |k-j|> \alpha_n    \end{subarray}} \| \EE\{s'_{\Delta_{n,j}} {(s'_{\Delta_{n,k}})}^\top\} \| \\
  & \leq c\,  |\mathcal I_n|^{-1}  \sum_{j \in \mathcal I_n} \sum_{
  \begin{subarray}{c}
  k \in \mathcal I_n \\ |k-j|> \alpha_n    \end{subarray}}  \frac{1}{|k-j|^{\gamma_2}} \leq c \sum_{|i|>\alpha_n} \frac{1}{|i|^{\gamma_2}} = o(1).  \end{align*}
  Finally $T_3\to 0$ from (ii) in Lemma~\ref{lem:cov}, which concludes the proof of condition~(c) of Theorem~\ref{tcl1}.

To prove assumption~(d), we apply the conditional GNZ formula \eqref{condGNZ} to write,  for any $j\in \mathcal I_n$, 
\begin{align*}
\EE \big( Z_{n,j} \mid \,& \bX_{\Delta_{n,k}},k\neq j \big) \\
&=
\EE  \left[\int_{\Delta_{n,j}} t(u,\bX_{B_{u,n}}) \left\{ 
\lth(u,\bX)-\lth(u,\bX_{B_{u,n}})
\right\} \dd u \mid \, \bX_{\Delta_{n,k}}, k\neq j \right]
\\
& \quad - \EE  \int_{\Delta_{n,j}} t(u,\bX_{B_{u,n}})\left\{ 
\lth(u,\bX)-\lth(u,\bX_{B_{u,n}})
\right\} \dd u .
\end{align*}
From Lemma~\ref{lem:XWn}, we have for any $u\in \RR^d$ and $\bx \in \Omega$
\[
  \|t(u,\bx_{B_{u,n}})\|\, |
  \lth(u, \bx) -\lth(u,\bx_{B_{u,n}}) | 
  \leq \frac{c}{\alpha_n^{\gamma^\prime}} Y(u,\bx)
\]
where $Y(u,\bx)=\|\{|t_m|(u,\bx)\}_{m\geq 1}\|H(u,\bx)e^{c\, G(u,\bx)}\lth(u,\bx)$ using the notation of the lemma. 
Since $|\Delta_{n,j}|\leq 1$, we deduce from the stationarity of $\bX$ and Lemma~\ref{lem:int} that
\[
  \EE \left\|
\EE \big( Z_{n,j} \mid \,\bX_{\Delta_{n,k}}, k\neq j \big)
\right\| \leq \frac{c}{\alpha_n^{\gamma^\prime}} \EE \{Y(0,\bX)\} = \mathcal O(\alpha_n^{-\gamma^\prime}).
\]
Hence
\[
 |\mathcal I_n|^{-1/2}\sum_{j\in\mathcal I_n}\EE\| \EE 
 \big( Z_{n,j} \mid \, \bX_{\Delta_{n,k}}, k\neq j
 \big)\| 
 = \mathcal O(\rho_n)
\]
tends to 0 from~\eqref{eq:rhon}. All conditions of Theorem~\ref{tcl1} are therefore satisfied, which yields that $\Sigma_n^{-1/2} S_n \overset{d}{\to} \mathcal N(0,I_p)$. The convergence $|W_n|^{-1} \Sigma_n\to \Sigma_\infty$ is an immediate consequence of  assumption (c) checked above. It remains to prove that \linebreak$\Sigma_n^{-1/2} \EE(s_{\Wna})\to 0$. This is a consequence of the GNZ formula, \eqref{ell},  Lemma~\ref{lem:XWn} and the condition~\eqref{eq:a} since
\begin{align*}
\| \Sigma_n^{-1/2} \EE(s_{\Wna}) \| & \leq \| \Sigma_n^{-1/2} \| \, \| \EE (s_{\Wna}) \| \\
&\leq c\, |\mathcal I_n|^{-1/2} \alpha_n^{-\gamma^\prime} \EE \int_{\Wna} Y(u,\bX) \dd u  =\mathcal O \left(\rho_n \right)=o(1).
\end{align*}

(ii) It is worth repeating that $\theta \to -\widetilde\LPL_{\Wna,R_n}(\bx;\theta)$ is a convex function with Hessian matrix  given by \eqref{eq:derivScore}. Following Lemmas~\ref{lem:XWn}-\ref{lem:int} and arguments developed in the proof of Proposition~\ref{prop:cvps}, we leave the reader to check that almost surely
\[
|\Wna|^{-1} \left\{ \frac{\dd}{\dd \theta\dd \theta^\top} \widetilde \LPL_{\Wna,R_n}(\bX;\theta) - 
\frac{\dd}{\dd \theta\dd \theta^\top}  \LPL_{\Wna}(\bX;\theta) \right\} \to 0  
\]
and
\[
 - |\Wna|^{-1} \frac{\dd}{\dd \theta\dd \theta^\top}  \LPL_{\Wna}(\bX;\theta) \to \EE
 \left\{t(0,\bX)t(0,\bX)^\top \lambda_\theta(0,\bX) \right\}
\]
as $n\to \infty$, which equals to $U_\infty$ when $\theta=\theta^\star$. We also note that~\eqref{eq:ident} implies that $U_\infty$ is a positive definite matrix. These facts and (i) allow us to apply \citet[Theorem 3.4.5]{guyon:95} to deduce the result.
\end{proof}

The following proposition focuses on the maximum logistic regression and states its strong consistency and asymptotic  normality.  The result is given without proof, but we claim that it follows by the same arguments as those involved in the proofs of  Proposition~\ref{prop:cvps} and Theorem~\ref{thm:cvloi}.

\begin{proposition} \label{prop:logistic}
  Under the assumptions of Proposition~\ref{prop:cvps}, the maximum logistic regression estimator defined by
  \[
\widehat\theta_{\scriptsize\widetilde{\LRL}} = \mathrm{argmax}_{\theta\in \Theta} \widetilde{\LRL}_{\Wna}(\bX;\theta)  
\] 
converges almost surely to $\theta^\star$ as $n\to \infty$ and under the assumptions of Theorem~\ref{thm:cvloi} it satisfies the following convergence in distribution
\[
  |W_n|^{1/2} \left(\widehat\theta_{\scriptsize\widetilde{\LRL}}-\theta^\star \right) \stackrel{d}{\to} \mathcal N \left( 0, V_{\infty}^{-1} \Gamma_\infty V_\infty^{-1} \right)
\]
where denoting $h(u,\bx)=\rho \,t(u,\bx)/\{\lth(u,\bx)+\rho\}$ for any $u\in \RR^d, \bx\in \Omega$,
\begin{align*}
\Gamma_\infty &=  \EE \left\{ h(0,\bX) h(0,\bX)^\top \lth(0,\bX)
\right\} \\
&\qquad + \int_{\RR^d} \EE  \left\{ 
h(0,\bX)h(v,\bX)^\top \lth(0,\bX)\lth(v,\bX)\right\}  \{ 1-e^{-\Phi_{\theta^\star}(v)} \}\dd v\nonumber\\
&\qquad+ \int_{\RR^d} \EE  \left\{  \lth (0,\bX)
\lth (v,\bX)    \Delta_v h(0,\bX)\Delta_0 h(v,\bX)^\top \right\}e^{-\Phi_{\theta^\star}(v)}\dd v \\
V_\infty &= \frac1\rho \EE \left\{ h(0,\bX)h(0,\bX)^\top \lth(0,\bX)
\right\}
\end{align*}
with $\Delta_v h(u,\bx)= h(u,\bx\cup v)-h(u,\bx)$ for any $u,v\in \RR^d$ and $\bx\in \Omega$.
\end{proposition}


\section{Simulation study} \label{sec:sim}

In this section, we present results of simulation experiments assessing the performance estimation of maximum pseudolikelihood estimators in the Lennard-Jones model, given by
\begin{equation}
  \label{LJ}
  \log \lambda_\theta(u,\bx) = \log(\beta) - \sum_{v \in \bx } \Phi (v-u) 
  \quad \mbox{ with } \quad
  \Phi(u) = 4\varepsilon \left\{ \left(\frac{\sigma}{\|u\|} \right)^6 - \left(\frac{\sigma}{\|u\|} \right)^{12} \right\}.
\end{equation}
We chose $\beta=100$ and $\sigma=0.1$ and considered three cases where $\varepsilon$ takes the values $0.1, 0.5$ and $1$ respectively, which, following \citet{baddeley:dereudre:13} we call low, moderate and high rigidity models. The realizations are generated using the Metropolis-Hastings algorithm, implemented in the \texttt{R} package \texttt{spatstat} \citep{baddeley:turner:05,baddeley:rubak:turner:15},  
on $W_n=[-n,n]^2$ and for $n=1/2,1,2$. To take into account the infinite range characteristic of the Lennard-Jones model, the processes are simulated on $[-n-2,n+2]^2$ and then clipped to $W_n$. Figure~\ref{fig:LJ} depicts some typical realizations on $[-1,1]^2$. 

For each model and each observation window, we considered three versions of maximum pseudolikelihood estimators given by~\eqref{lpl} of the parameter vector $\{\log(\beta),\sigma,\varepsilon\}^\top$: (i) $\alpha_n=R_n \in [0.05,0.3]$, (ii) $\alpha_n \in [0.05,0.3]$ and $R_n=\infty$, (iii) $\alpha_n=0$, $R_n=\infty$. We remind that the values $\alpha_n=0$ and $R_n=\infty$ respectively mean that no border erosion is considered (i.e. $W_n \ominus \alpha_n = W_n$) and the maximal possible range of interaction in $W_n$ is taken into account
(i.e. $\lambda(u,\bx_{u,R_n}) = \lambda(u,\bx_{W_n})$). Writing $\alpha_n$ and/or $R_n \in [0.05,0.3]$ means that we evaluated the estimates for 30 values regularly sampled in $[0.05,0.3]$.

We computed the pseudolikelihood estimator by using a $100\times 100$ grid of quadrature points to discretize the integral involved in~\eqref{lpl}. 
We did not use the  Berman-Turner approximation, implemented in \texttt{spatstat} for a large class of models excluding~\eqref{LJ} (see \citet{baddeley:turner:00}), because the latter may artificially lead to biased estimates for very repulsive patterns. 
As suggested by \cite{baddeley:dereudre:13}, to minimise numerical problems (overflow, instability, slow convergence) we rescaled the interpoint distances to a unit equal to the true value of~$\sigma$.

We define the weighted mean squared error $\mathrm{WMSE}$ by
\[
  \mathrm{WMSE} = 
\frac{\widehat E \left\{ \left(\widehat{\log \beta} - \log \beta\right)^2\right\}}{(\log \beta)^2} +
\frac{\widehat E \left\{ \left(\widehat{\sigma} - \sigma\right)^2\right\}}{\sigma^2} +
\frac{\widehat E \left\{ \left(\widehat{\varepsilon} - \varepsilon\right)^2\right\}}{\varepsilon^2}
\]
and we consider in the following its root  $\mathrm{RWMSE}=\sqrt{\mathrm{WMSE}}$. Similarly, we define the root-weighted squared bias  and the root-weighted variance respectively denoted by $\mathrm{RWSQ}$ and $\mathrm{RWV}$.

Tables~\ref{tab:rmse} and \ref{tab:bias} summarize the simulation study based on 100 replications, where we report the values of ${\mathrm{RWMSE}}$, ${\mathrm{RWSB}}$ and ${\mathrm{RWV}}$.
When $\alpha_n$ and/or $R_n$ vary, we report in Table~\ref{tab:rmse} the smallest value of  $\mathrm{RWMSE}$ and the associated  value $\alpha_{\mathrm{opt}}$ of $\alpha_n$ between brackets. To be consistent, we report in  Table~\ref{tab:bias}  the values of ${\mathrm{RWSB}}$ and ${\mathrm{RWV}}$ associated to $\alpha_{\mathrm{opt}}$.
We observe that the three versions of the estimates have a $\mathrm{RWMSE}$ decreasing with $n$ for the three Lennard-Jones models. In case (i) where $\alpha_n=R_n$, the optimal value seems to be around $\alpha_n=0.15$. A closer look at the estimates showed us that their average behavior (sample mean and standard deviation) fluctuate quite a lot with $\alpha_n$. In case (ii) where $R_n=\infty$, we observed that the biases of the estimates do not fluctuate that much with $\alpha_n$. Since the estimates had smaller standard deviation when the amount of information is maximal, i.e. when $\alpha_n$ is low, this explains why the smallest value of $\alpha_n$ led in almost all cases to the smallest ${\mathrm{RWMSE}}$. Surprisingly, the third situation corresponding to $\alpha_n=0$ and $R_n=\infty$ produced very interesting results which are optimal or close to the optimal ones in all cases considered. This estimator may be very time consuming to evaluate for very large datasets since all the points are involved in the evaluations of the Papangelou conditional intensity. Nonetheless, for the setting considered in this simulation study the computational time differences were negligible. The situation $\alpha_n=0$ and $R_n=\infty$ is supported by Proposition~\ref{prop:cvps} (consistency) but not by Theorem~\ref{thm:cvloi} (asymptotic normality).  However, the normal QQ-plots  in Figures~\ref{fig:qqnorm1}-\ref{fig:qqnorm3} seem to show a convergence to a Gaussian behavior for all our choices of $\alpha_n$ and $R_n$, with approximatively the same rate of convergence, i.e. $|W_n|^{-1/2}$, if we refer to the decreasing rate of the slopes in each QQ-plot. Note that the Gaussian behavior is less clear in the low rigidity Lennard-Jones model than in the moderate and high rigidity cases, but this seems specific to the model rather than to the estimators.
In conclusion, to estimate the parameters of a Lennard-Jones model using the pseudolikelihood method, we recommend to use no erosion and no finite range correction.

\begin{figure}[htbp]
  \begin{tabular}{lll}
  \includegraphics[scale=.5]{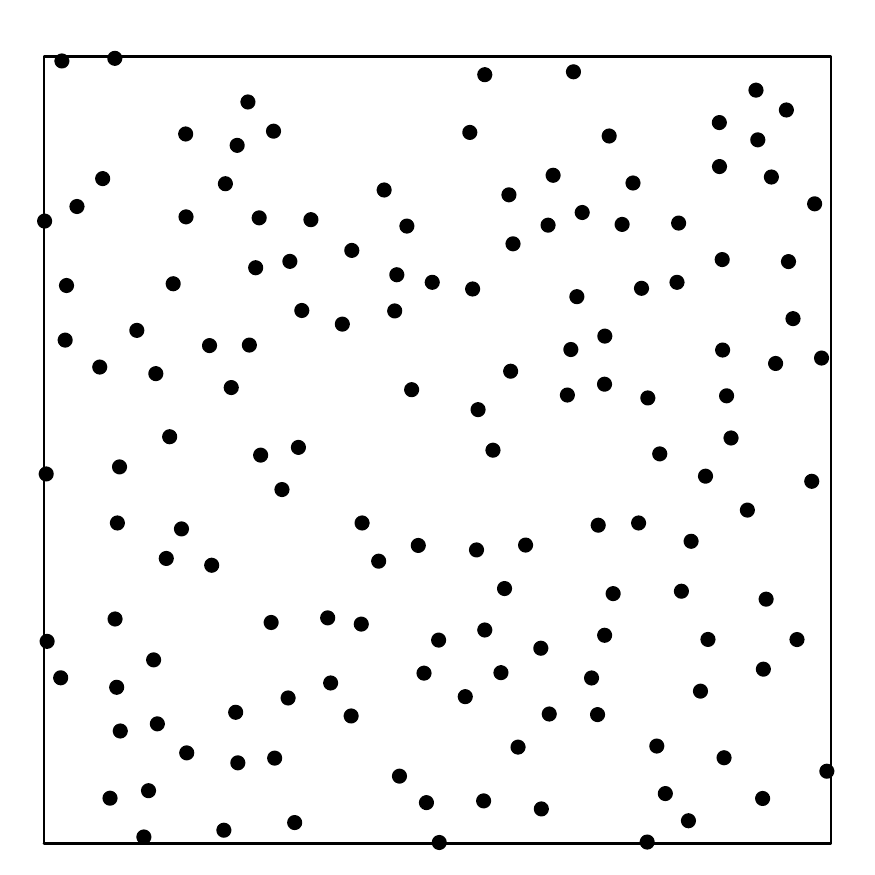}&\includegraphics[scale=.5]{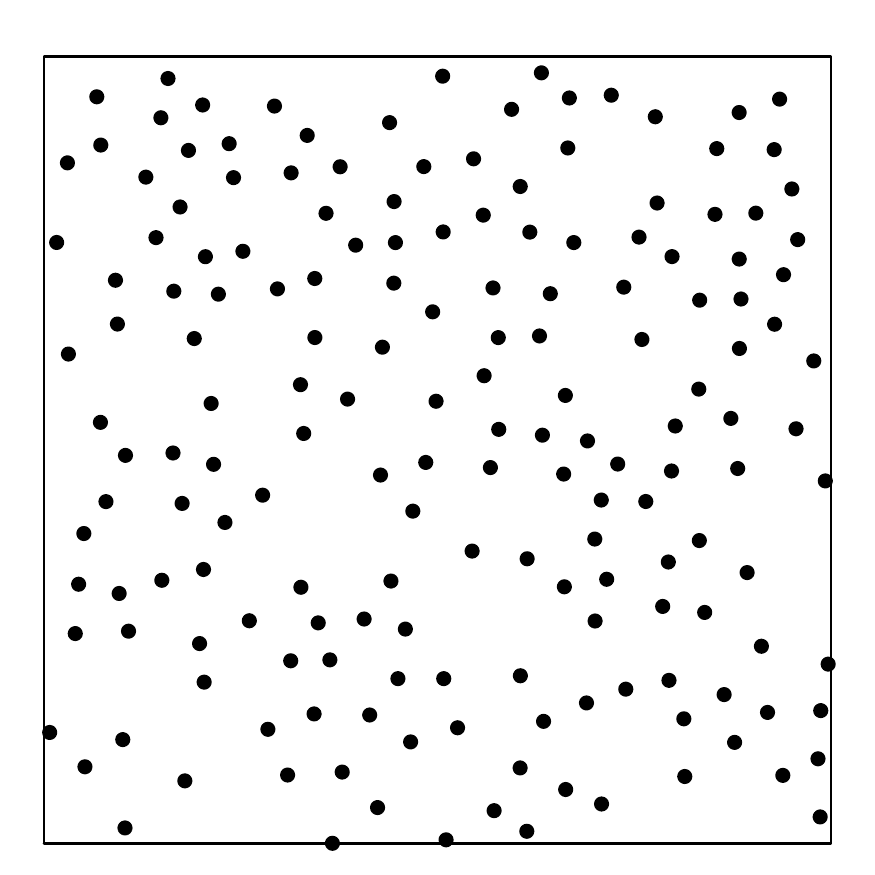}&\includegraphics[scale=.5]{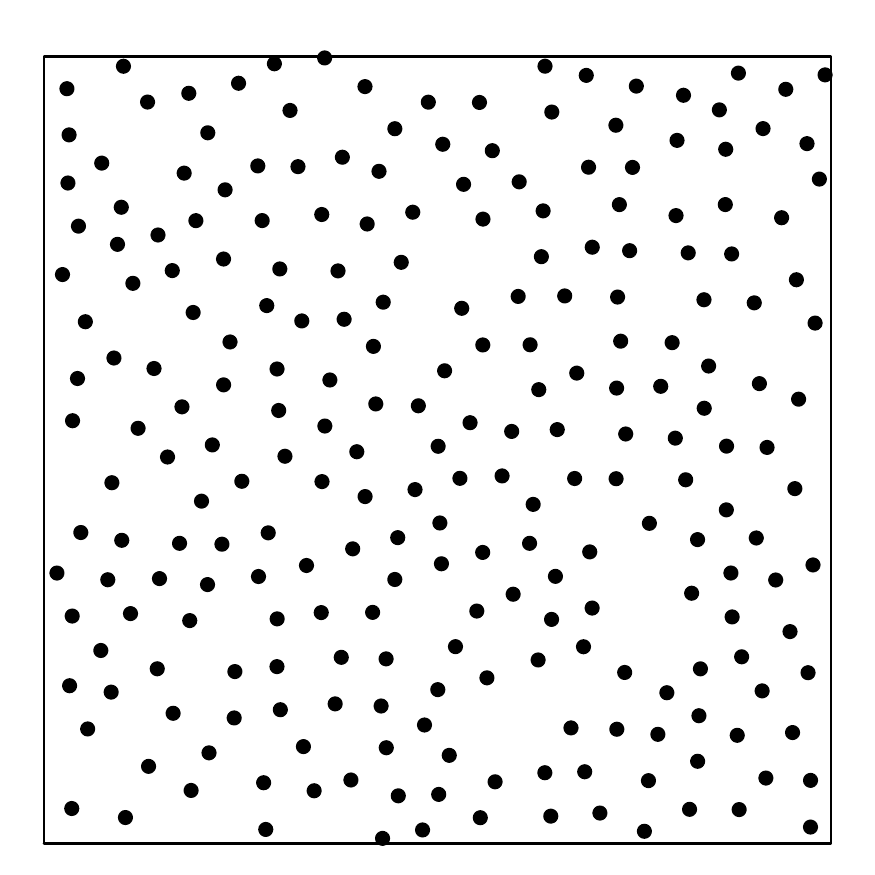}
  \end{tabular}
  \caption{\label{fig:LJ} Typical realizations  on $[-1,1]^2$ of a Lennard-Jones model with parameters $\beta=100$, $\sigma=0.1$ and $\varepsilon=0.1$ (left), 0.5 (middle) and 1 (right). }
\end{figure}

\begin{table}[htbp]
\centering
\begin{tabular}{rlll}
\hline
& \multicolumn{3}{c}{RWMSE}\\
&  $[-1/2,1/2]^2$ & $[-1,1]^2$& $[-2,2]^2$ \\
\hline
 \multicolumn{3}{l}{{\bf Low  ($\varepsilon=0.1$)}}\\
 $\alpha_n=R_n \in [0.05,0.3]$ & 3.26 (0.13) & 1.25 (0.13) & 0.62 (0.12) \\
$\alpha_n \in [0.05,0.3]$, $R_n=\infty$& 3.72 (0.05) & 1.79 (0.05) & 0.63 (0.06) \\
$\alpha_n=0$, $R_n=\infty$& 3.5  & 1.66  & 0.69  \\
   \hline
   \multicolumn{3}{l}{{\bf Moderate  ($\varepsilon=0.5$)}}\\
  $\alpha_n=R_n \in [0.05,0.3]$ & 0.65 (0.12) & 0.34 (0.14) & 0.2 (0.15) \\
 $\alpha_n \in [0.05,0.3]$, $R_n=\infty$& 0.68 (0.05) & 0.38 (0.05) & 0.19 (0.05) \\
 $\alpha_n=0$, $R_n=\infty$& 0.59  & 0.33  & 0.18 \\
\hline
 \multicolumn{3}{l}{{\bf High  ($\varepsilon=1$)}}\\
  $\alpha_n=R_n \in [0.05,0.3]$& 1.04 (0.08) & 0.42 (0.16) & 0.13 (0.16) \\
 $\alpha_n \in [0.05,0.3]$, $R_n=\infty$& 1.34 (0.05) & 0.36 (0.05) & 0.16 (0.05) \\
 $\alpha_n=0$, $R_n=\infty$& 1.23  & 0.27 & 0.17 \\
 \hline
\end{tabular} 
\caption{\label{tab:rmse} Root-weighted mean squared errors (RWMSE)  of parameters estimates for different Lennard-Jones models. The results are based on 100 replications. The realizations are generated on $[-n-2,n+2]^2$ for $n=1/2,1,2$ and the window of observation corresponds to $[-n,n]^2$. When it makes sense, we indicate between brackets the value $\alpha_{\mathrm{opt}}$ of $\alpha_n$ leading to the minimal value of RWMSE.
} 
\end{table}

\begin{table}[htbp]
\centering
\begin{tabular}{rrlrlrl}
\hline
& \multicolumn{6}{c}{ RWSB and RWV}  \\
&  \multicolumn{2}{c}{$[-1/2,1/2]^2$} & \multicolumn{2}{c}{$[-1,1]^2$}& 
\multicolumn{2}{c}{$[-2,2]^2$} \\
\hline
 \multicolumn{6}{l}{{\bf Low  ($\varepsilon=0.1$)}}\\
 $\alpha_n=R_n \in [0.05,0.3]$ & 1.82 & 2.70 & 0.57 & 1.11 & 0.07 & 0.62 \\
$\alpha_n \in [0.05,0.3]$, $R_n=\infty$ & 2.49 & 2.76 & 0.82 & 1.59 & 0.03 & 0.63 \\
$\alpha_n=0$, $R_n=\infty$ & 2.36 & 2.59 & 0.78 & 1.46 & 0.20 & 0.66 \\
   \hline
   \multicolumn{6}{l}{{\bf Moderate  ($\varepsilon=0.5$)}}\\
  $\alpha_n=R_n \in [0.05,0.3]$ & 0.23 & 0.60 & 0.16 & 0.30 & 0.07 & 0.19 \\
 $\alpha_n \in [0.05,0.3]$, $R_n=\infty$ & 0.04 & 0.66 & 0.10 & 0.37 & 0.02 & 0.19 \\
 $\alpha_n=0$, $R_n=\infty$& 0.07 & 0.58 & 0.02 & 0.33 & 0.02 & 0.18 \\
\hline
 \multicolumn{6}{l}{{\bf High  ($\varepsilon=1$)}}\\
  $\alpha_n=R_n \in [0.05,0.3]$& 0.43 & 0.71 & 0.16 & 0.39 & 0.07 & 0.12 \\
 $\alpha_n \in [0.05,0.3]$, $R_n=\infty$& 0.11 & 1.27 & 0.13 & 0.34 & 0.12 & 0.10 \\
 $\alpha_n=0$, $R_n=\infty$& 0.06 & 1.23 & 0.05 & 0.26 & 0.12 & 0.11 \\
 \hline
\end{tabular} 
\caption{\label{tab:bias} Root-weighted squared biases (RWSB) and variances (RWS) of parameters estimates for different Lennard-Jones models. 
The setting is the same as in Table~\ref{tab:rmse}. 
When $\alpha_n$ and $R_n$ vary, we report the values leading to the minimal $\mathrm{RWMSE}$, i.e. the values associated to  $\alpha_n=\alpha_{\mathrm{opt}}$ as given in Table~\ref{tab:rmse}.
}
\end{table}

\begin{figure}[htbp]
\begin{tabular}{ll}
  \includegraphics[scale=.45]{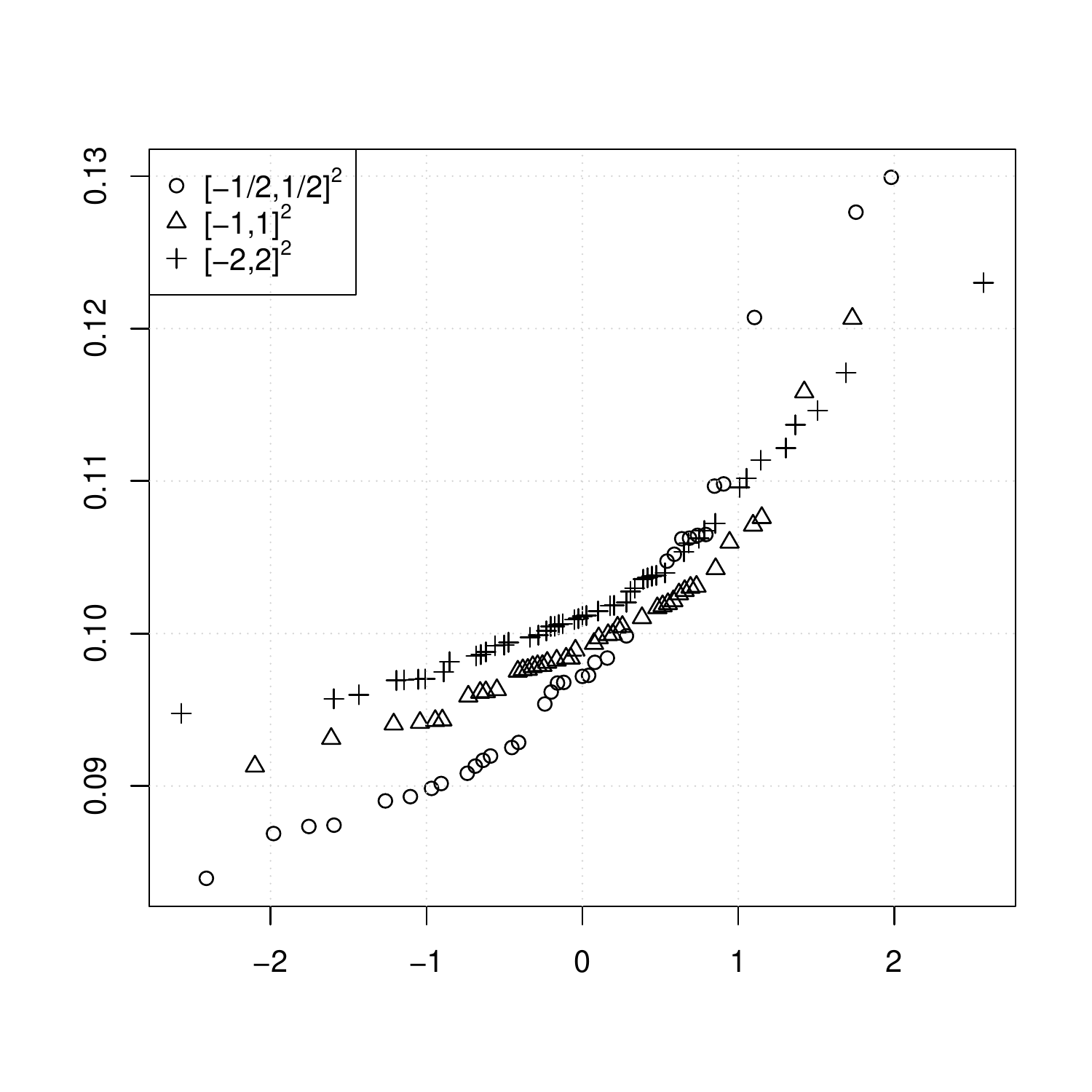} &\includegraphics[scale=.45]{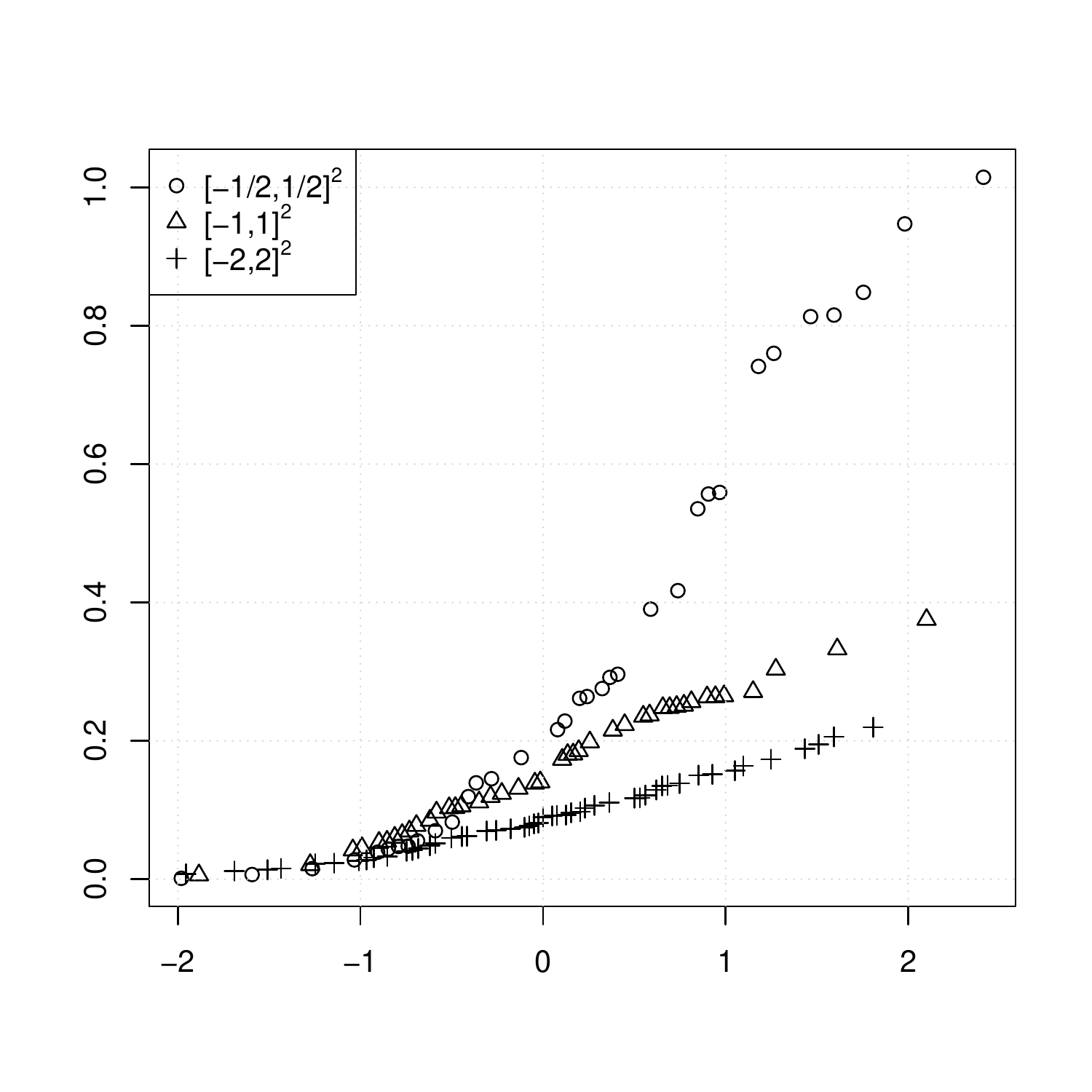} \\
  \includegraphics[scale=.45]{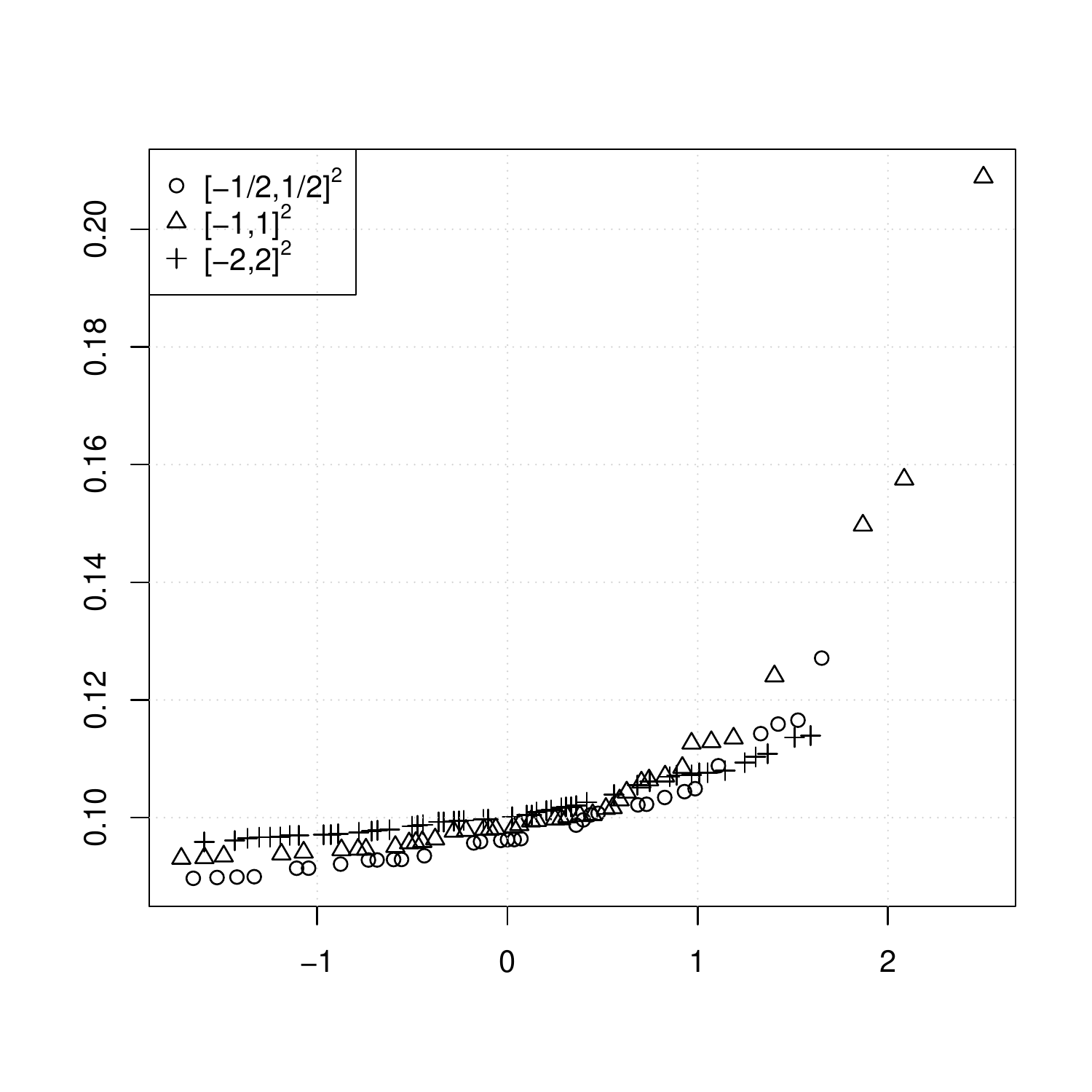} &\includegraphics[scale=.45]{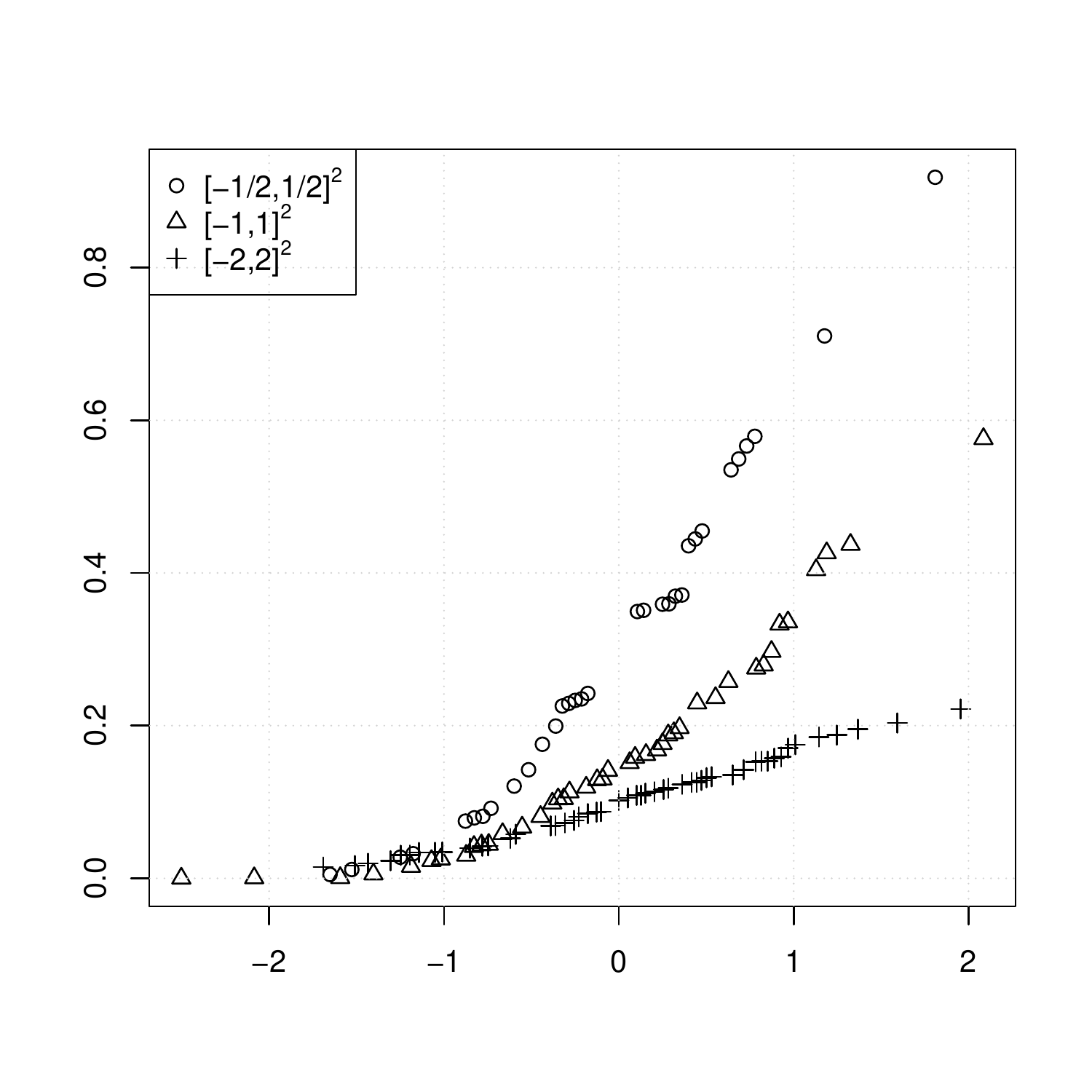} \\
  \includegraphics[scale=.45]{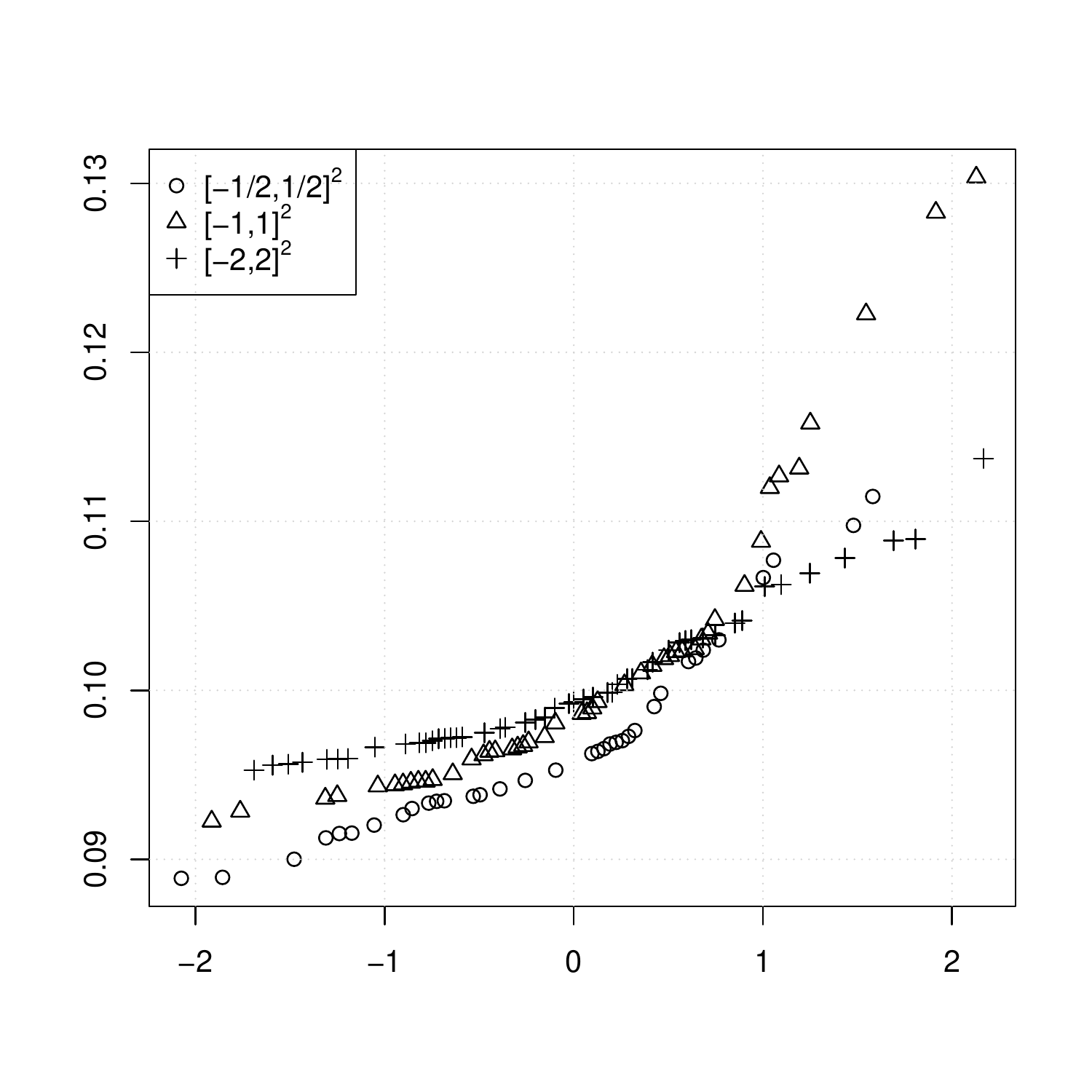} &\includegraphics[scale=.45]{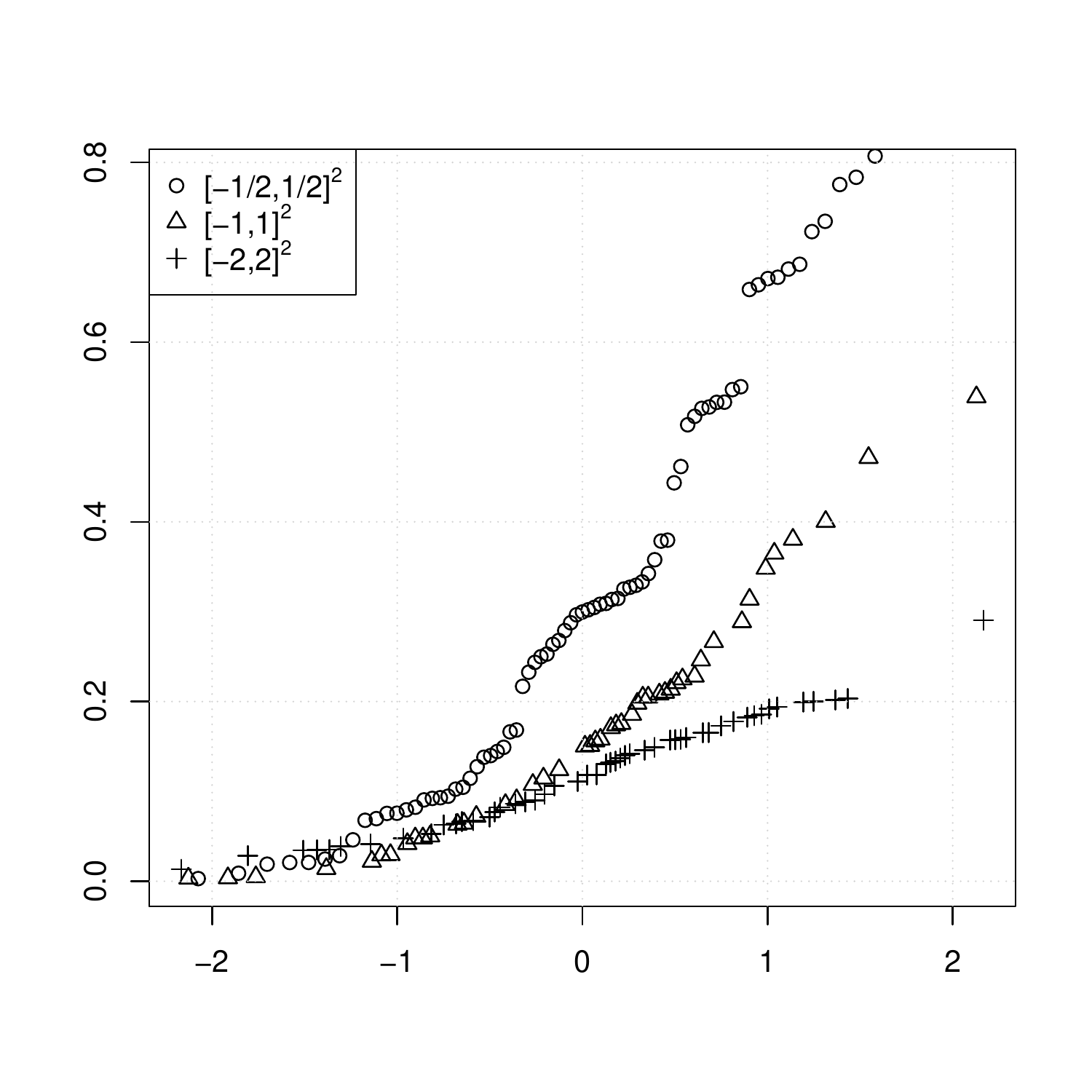} \\
\end{tabular}  
\caption{\label{fig:qqnorm1} Normal QQ-plots for estimates of the parameter $\varepsilon$ (left) and $\sigma$ (right) for the low interaction Lennard-Jones model, i.e. $(\log(\beta),\varepsilon,\sigma)=(\log(100),0.1,0.1)$. The first row (resp. the second and third) corresponds to estimates obtained with $\alpha_n=R_n=\alpha_{\mathrm{opt}}$ (resp. with $\alpha_n=\alpha_{\mathrm{opt}}, R_n=\infty$ and with $\alpha_n=0,R_n=\infty$). The optimal values  $\alpha_{\mathrm{opt}}$ are given in Table~\ref{tab:rmse}.}
\end{figure}

\begin{figure}[htbp]
\begin{tabular}{ll}
  \includegraphics[scale=.45]{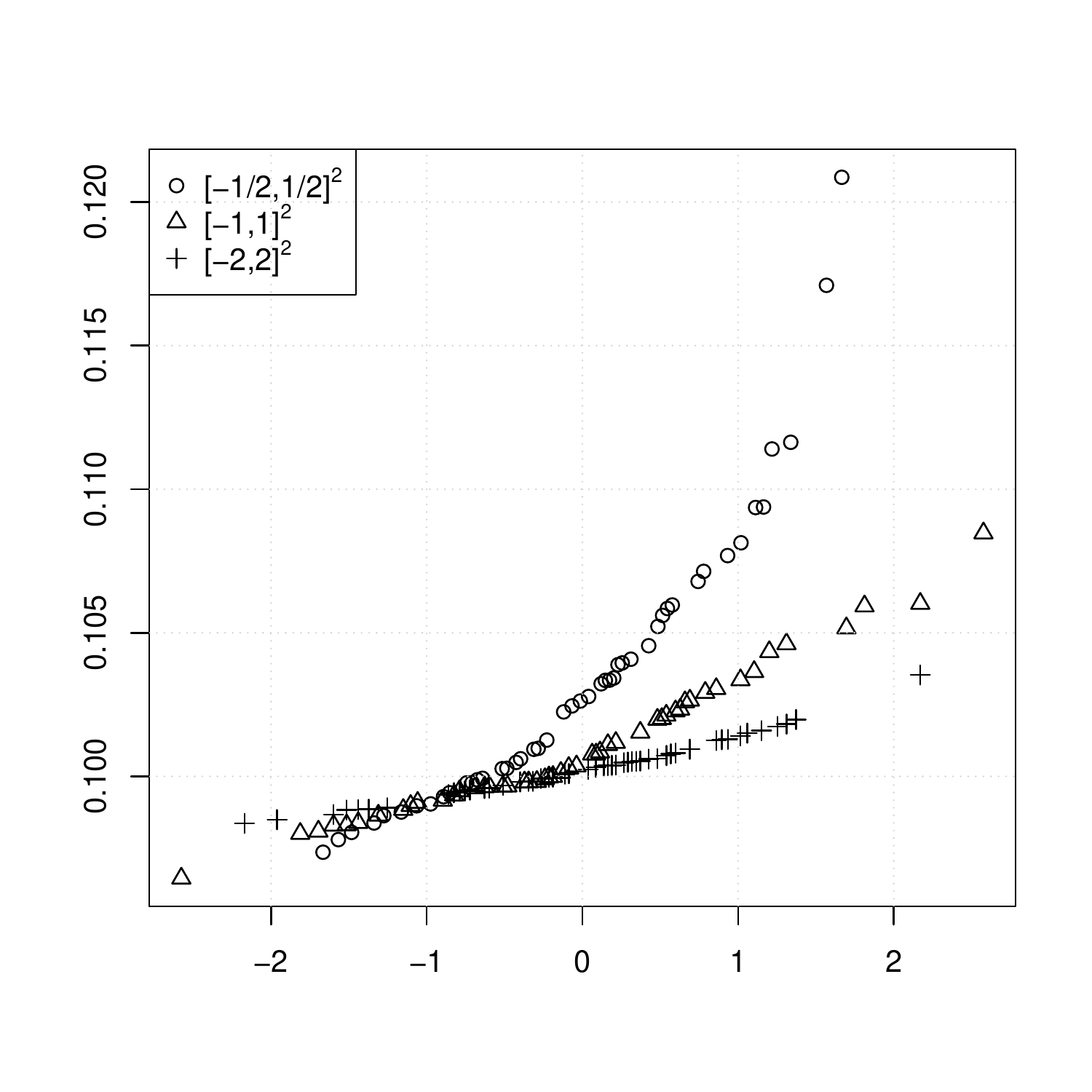} &\includegraphics[scale=.45]{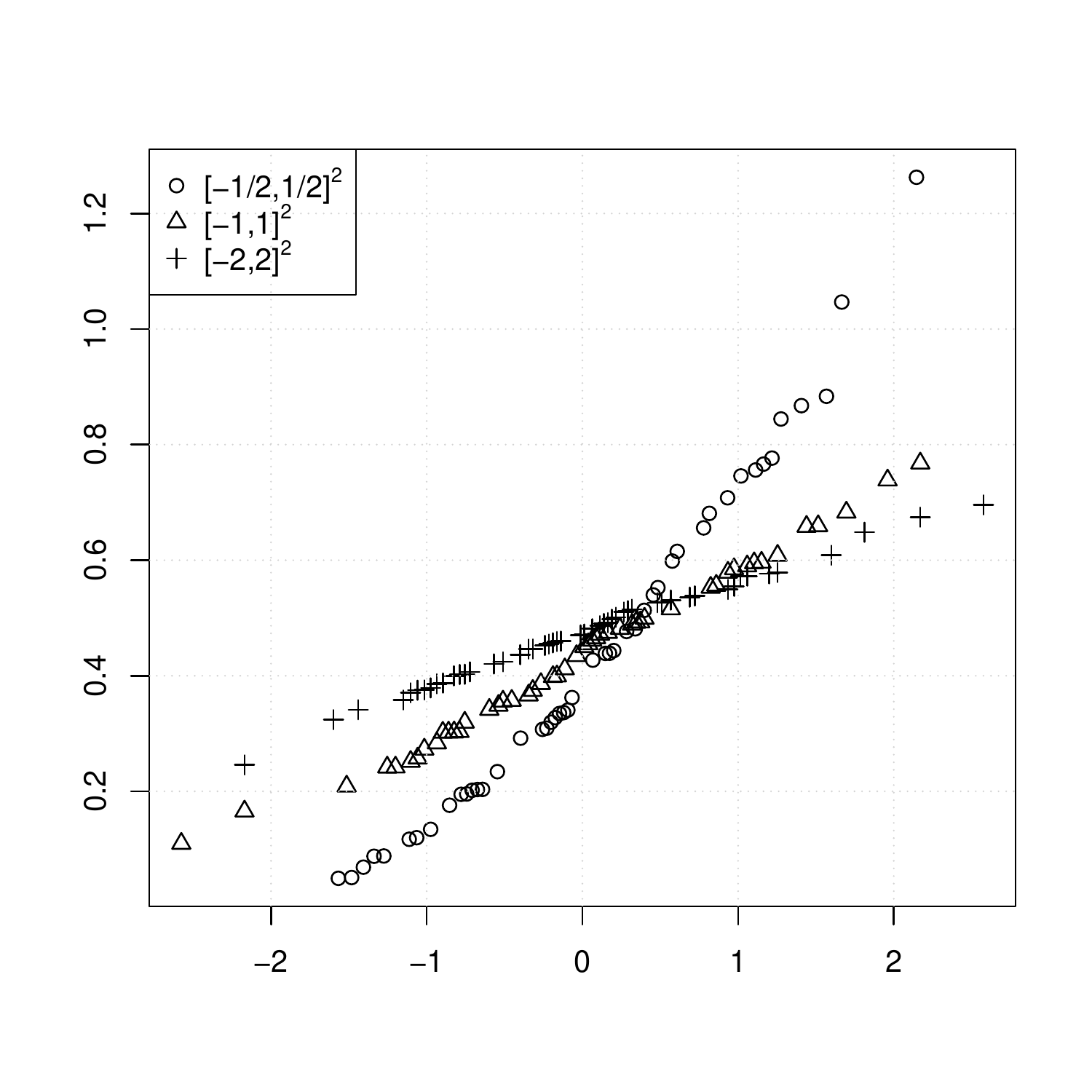} \\
  \includegraphics[scale=.45]{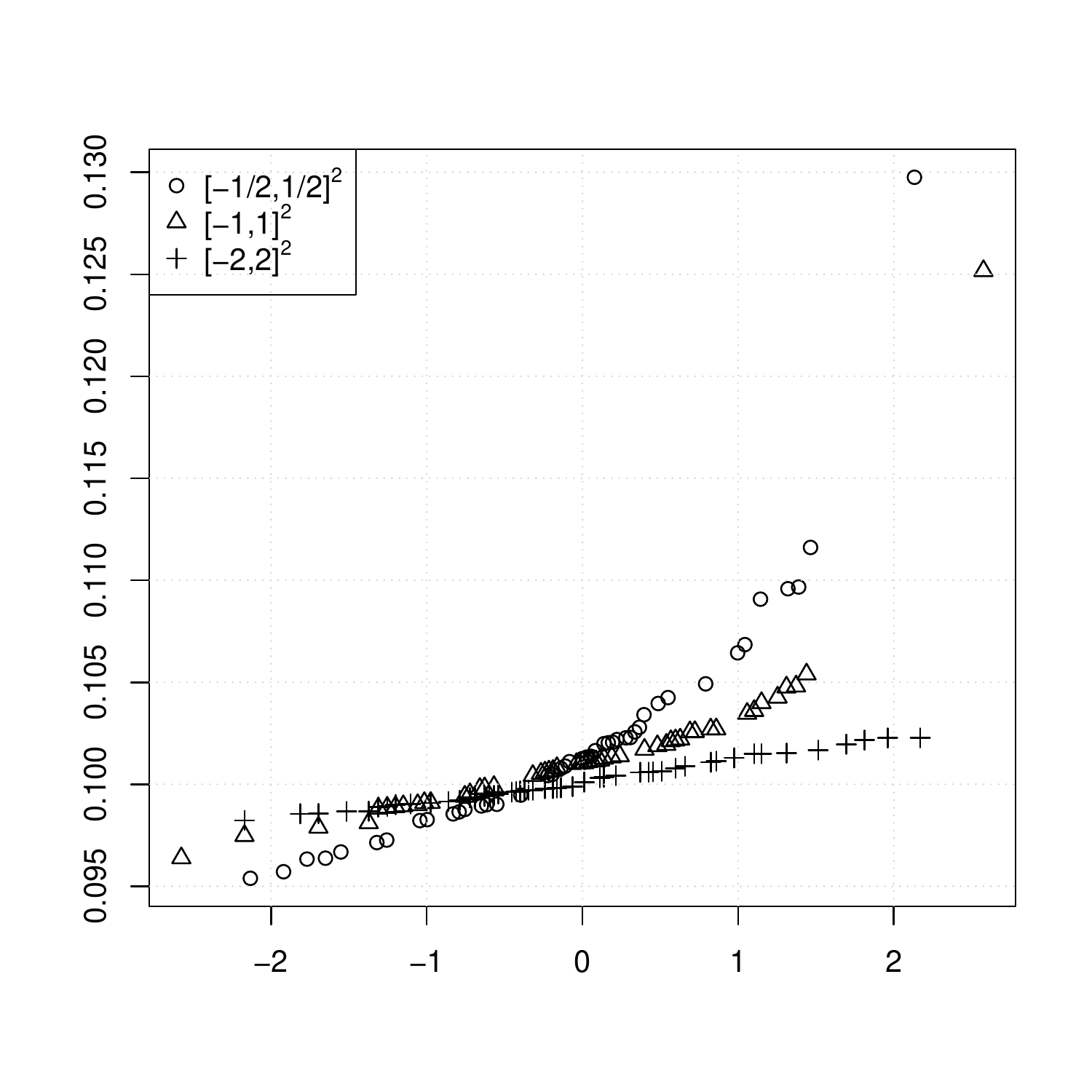} &\includegraphics[scale=.45]{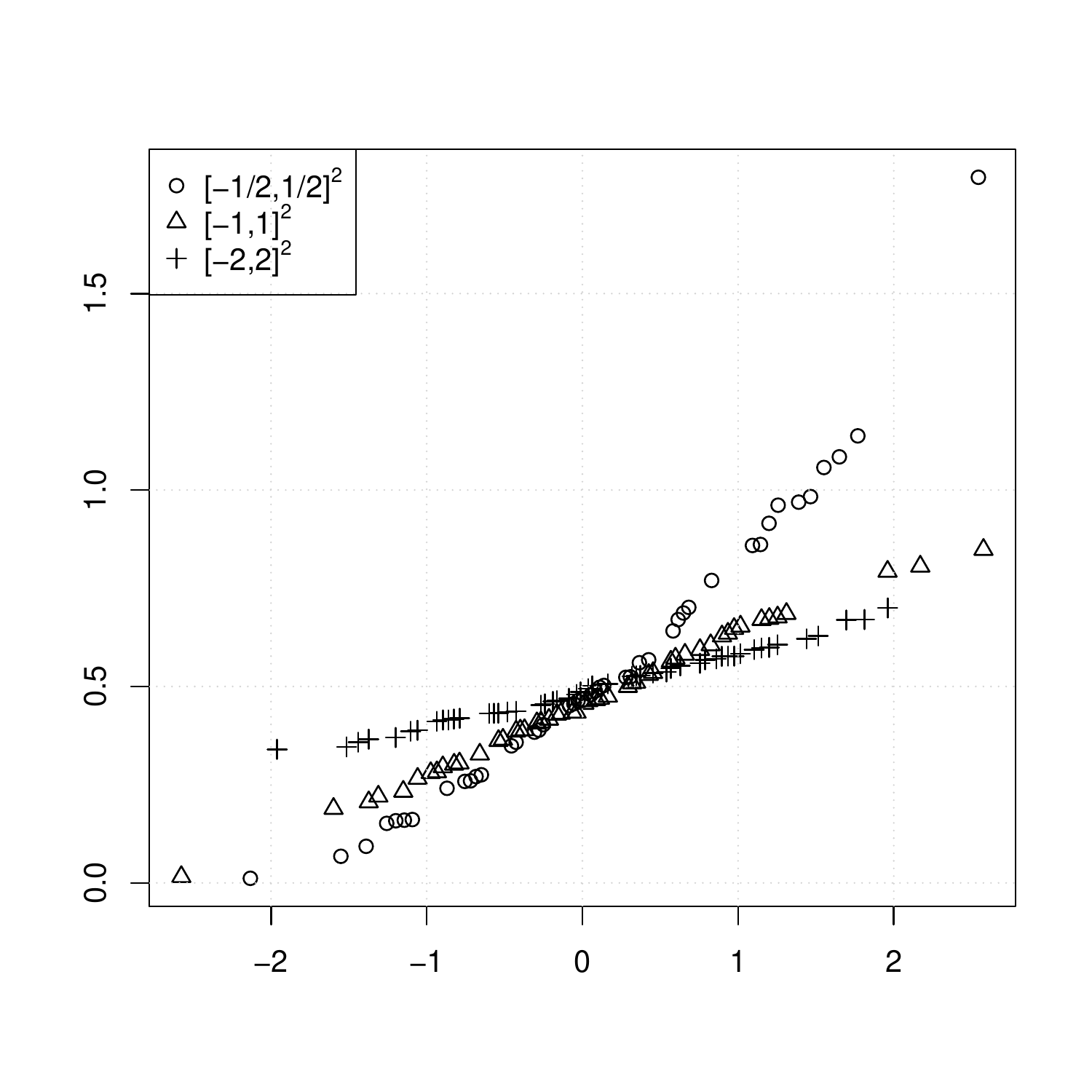} \\
  \includegraphics[scale=.45]{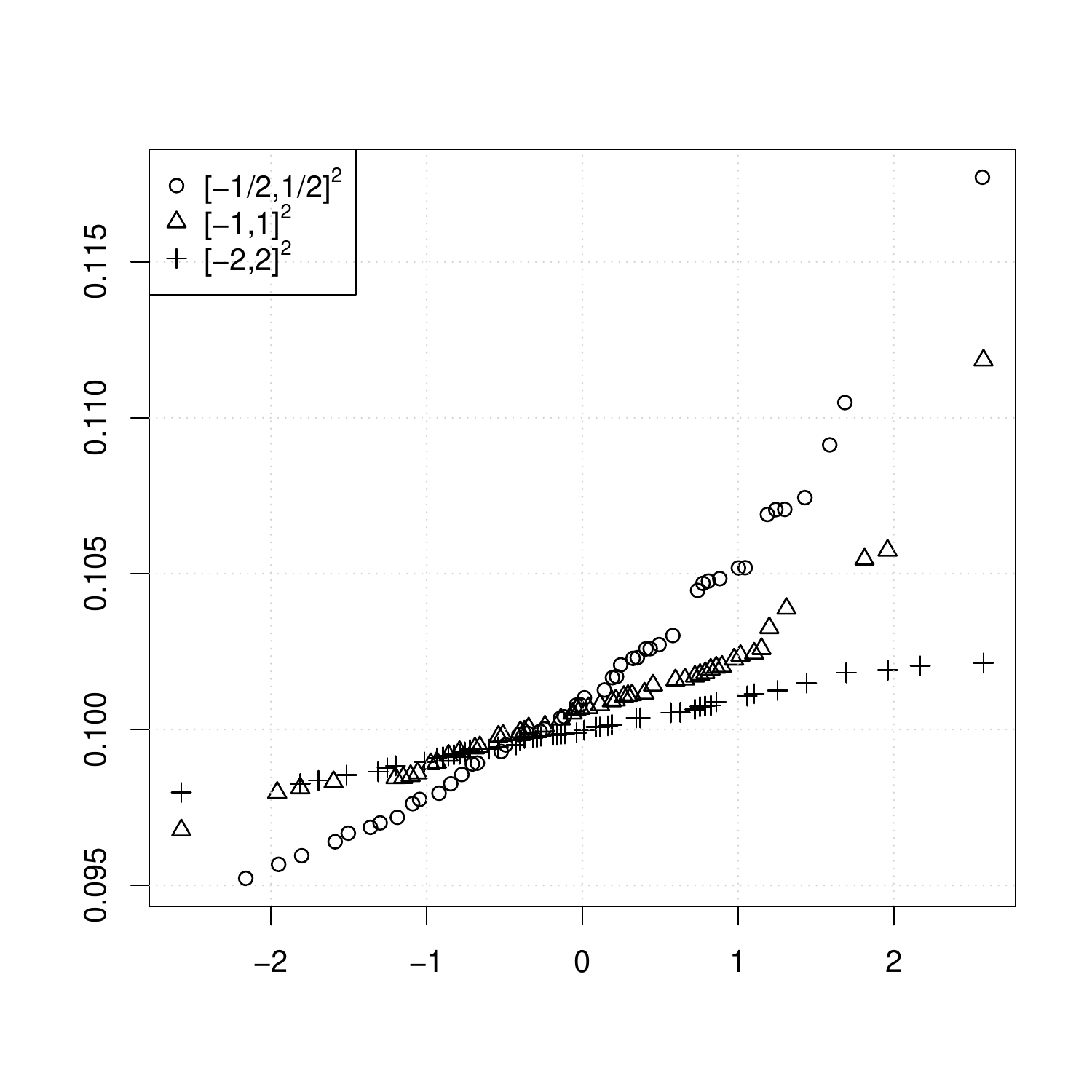} &\includegraphics[scale=.45]{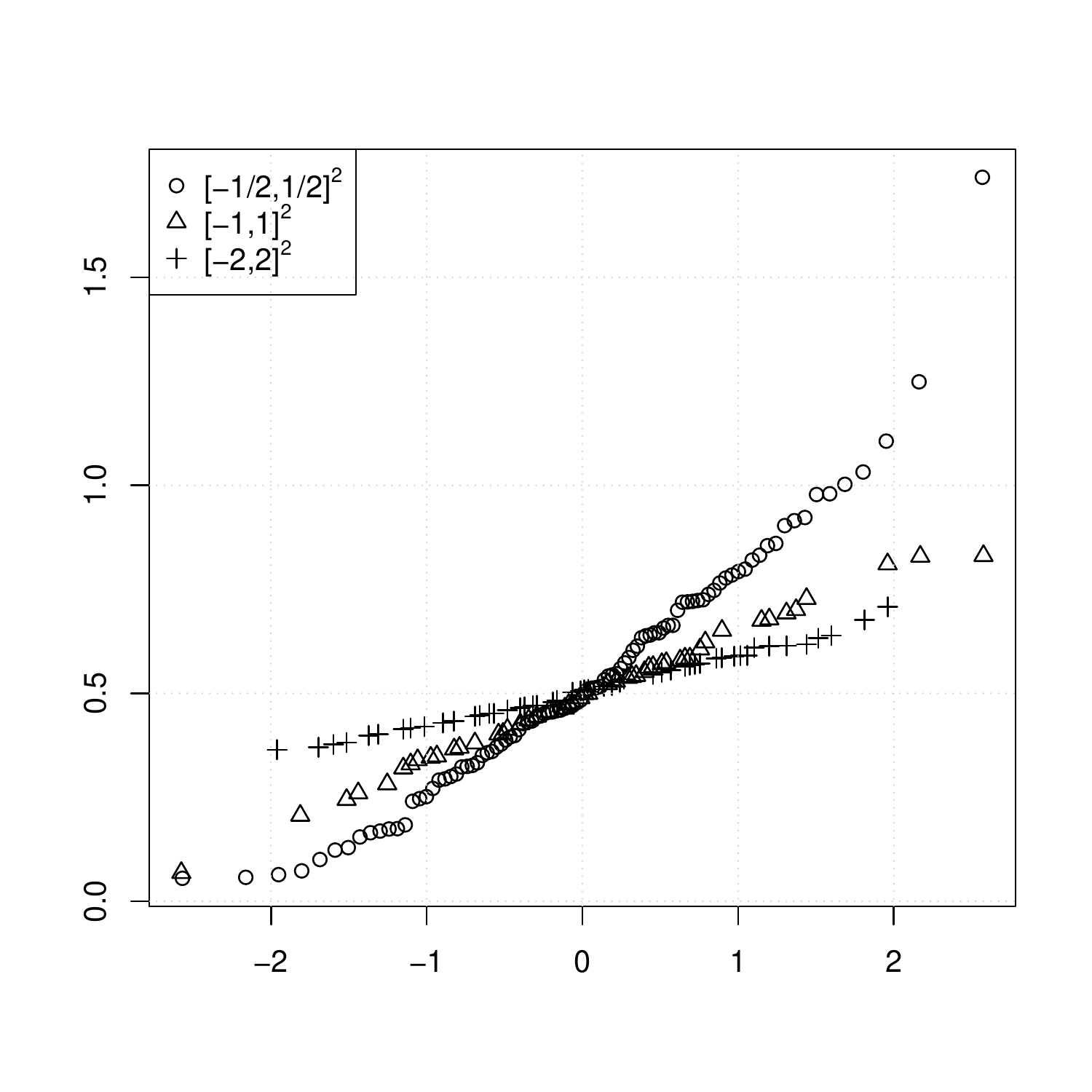} \\
\end{tabular}  
\caption{\label{fig:qqnorm2} Normal QQ-plots for estimates of the parameter $\varepsilon$ (left) and $\sigma$ (right) for the moderate interaction Lennard-Jones model, i.e. $(\log(\beta),\varepsilon,\sigma)=(\log(100),0.5,0.1)$. The first row (resp. the second and third) corresponds to estimates obtained with $\alpha_n=R_n=\alpha_{\mathrm{opt}}$ (resp. with $\alpha_n=\alpha_{\mathrm{opt}}, R_n=\infty$ and with $\alpha_n=0,R_n=\infty$). The optimal values  $\alpha_{\mathrm{opt}}$ are given in Table~\ref{tab:rmse}.}
\end{figure}

\begin{figure}[htbp]
\begin{tabular}{ll}
  \includegraphics[scale=.45]{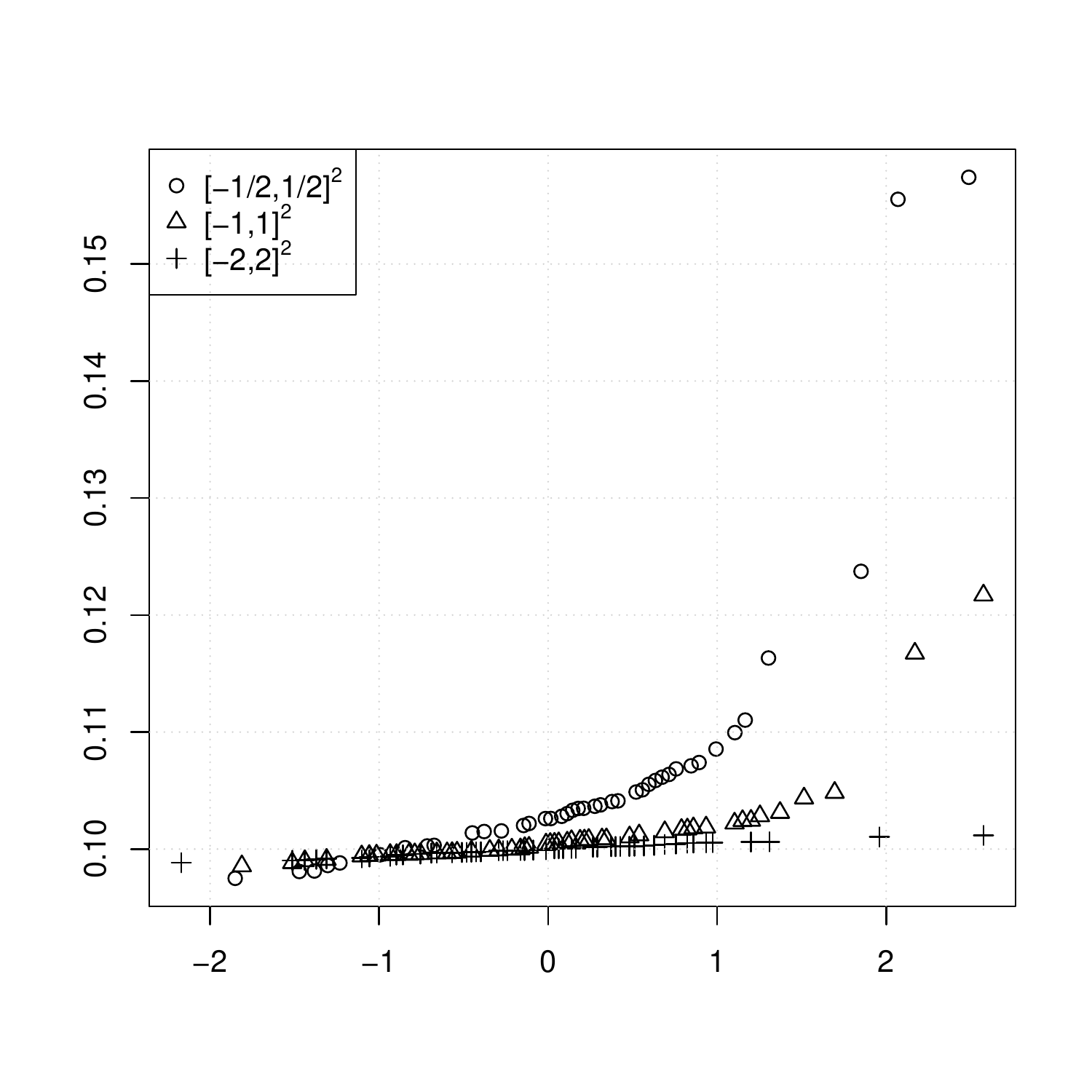} &\includegraphics[scale=.45]{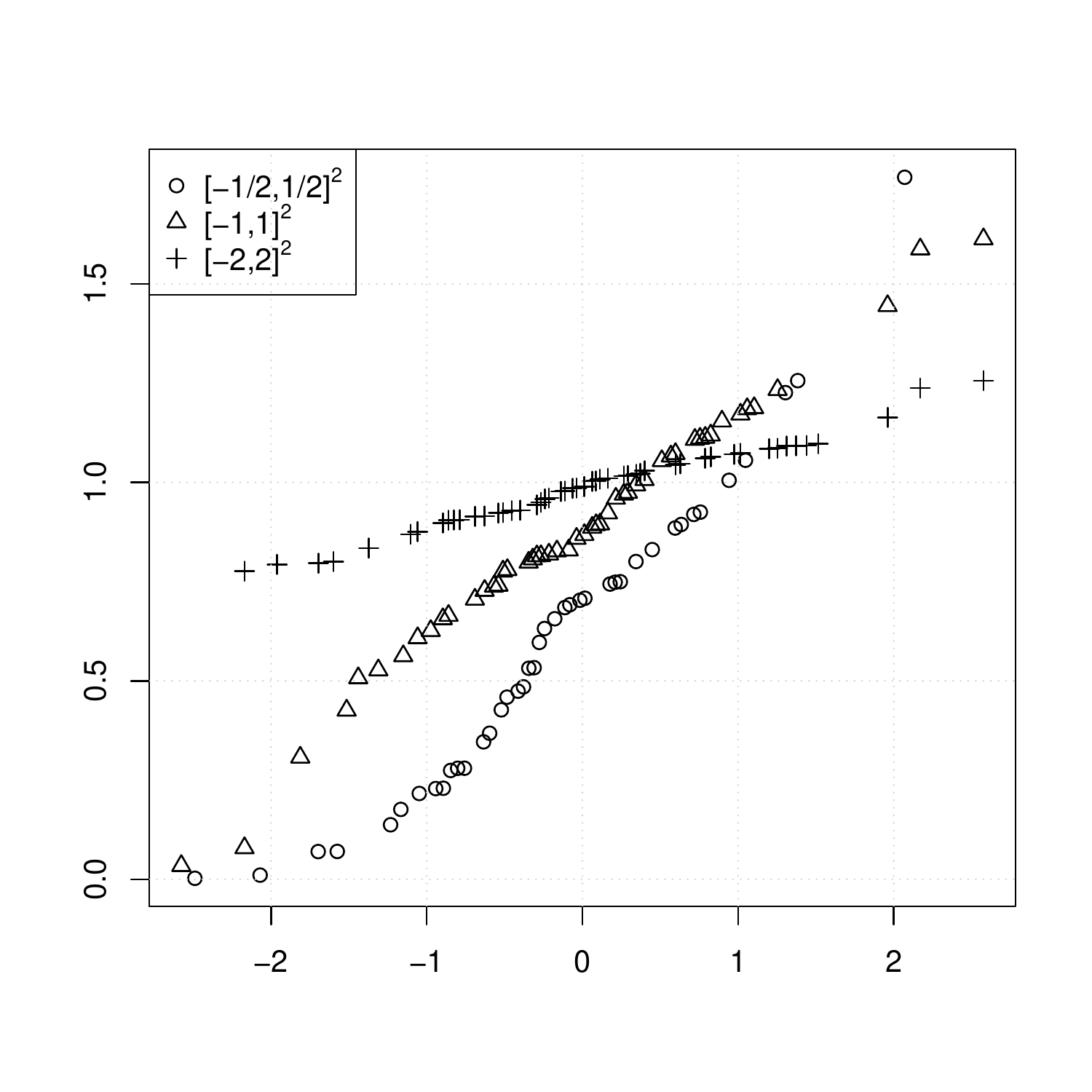} \\
  \includegraphics[scale=.45]{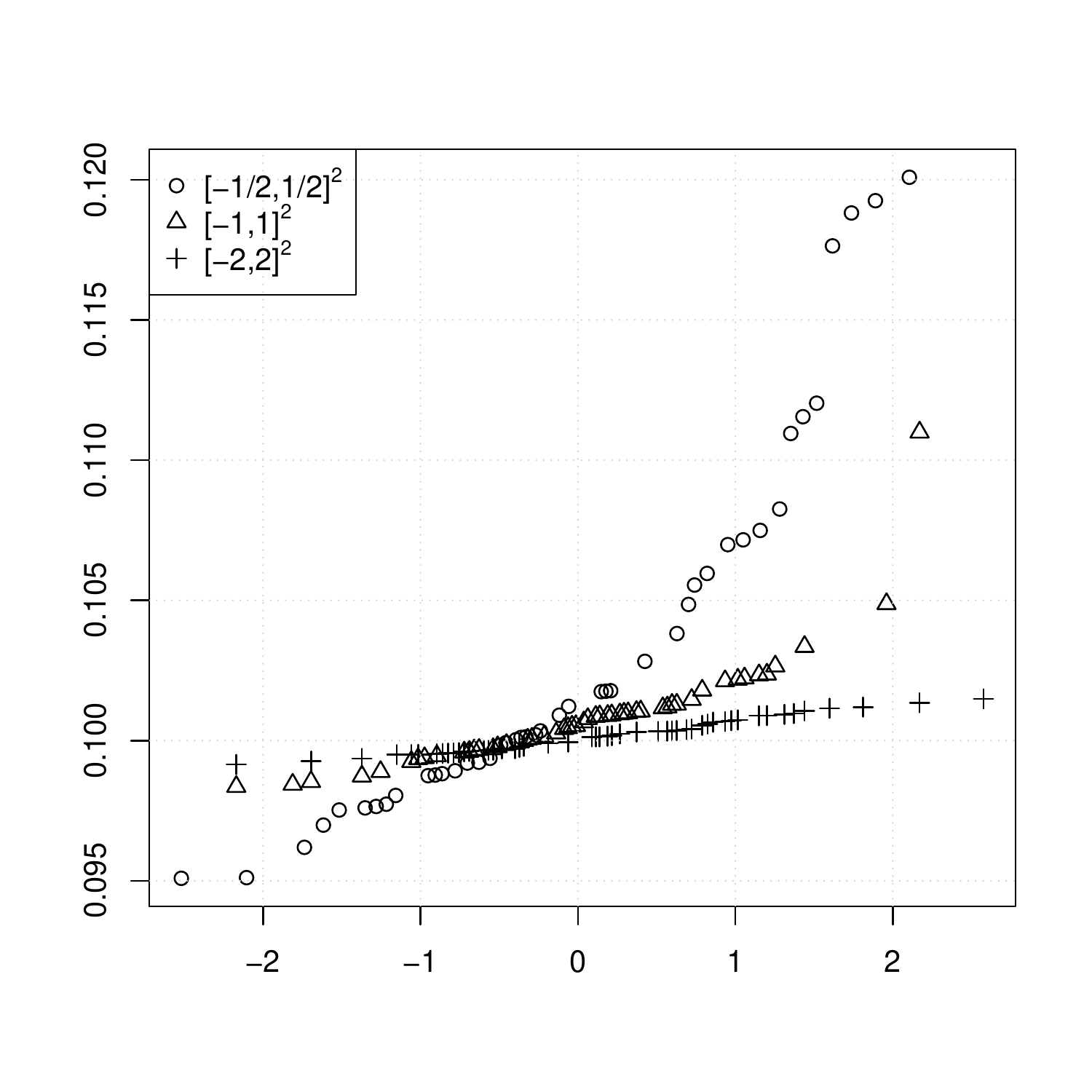} &\includegraphics[scale=.45]{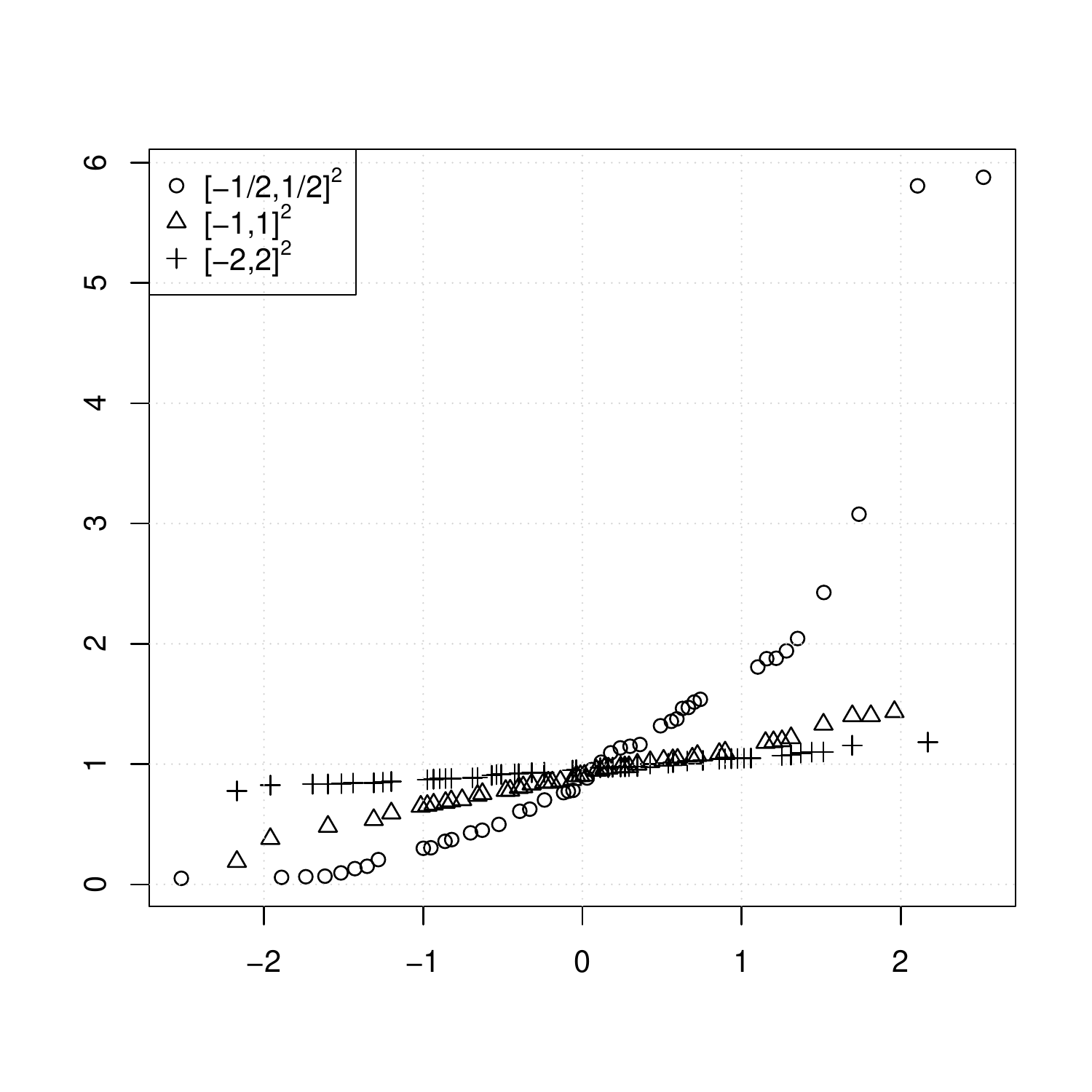} \\
  \includegraphics[scale=.45]{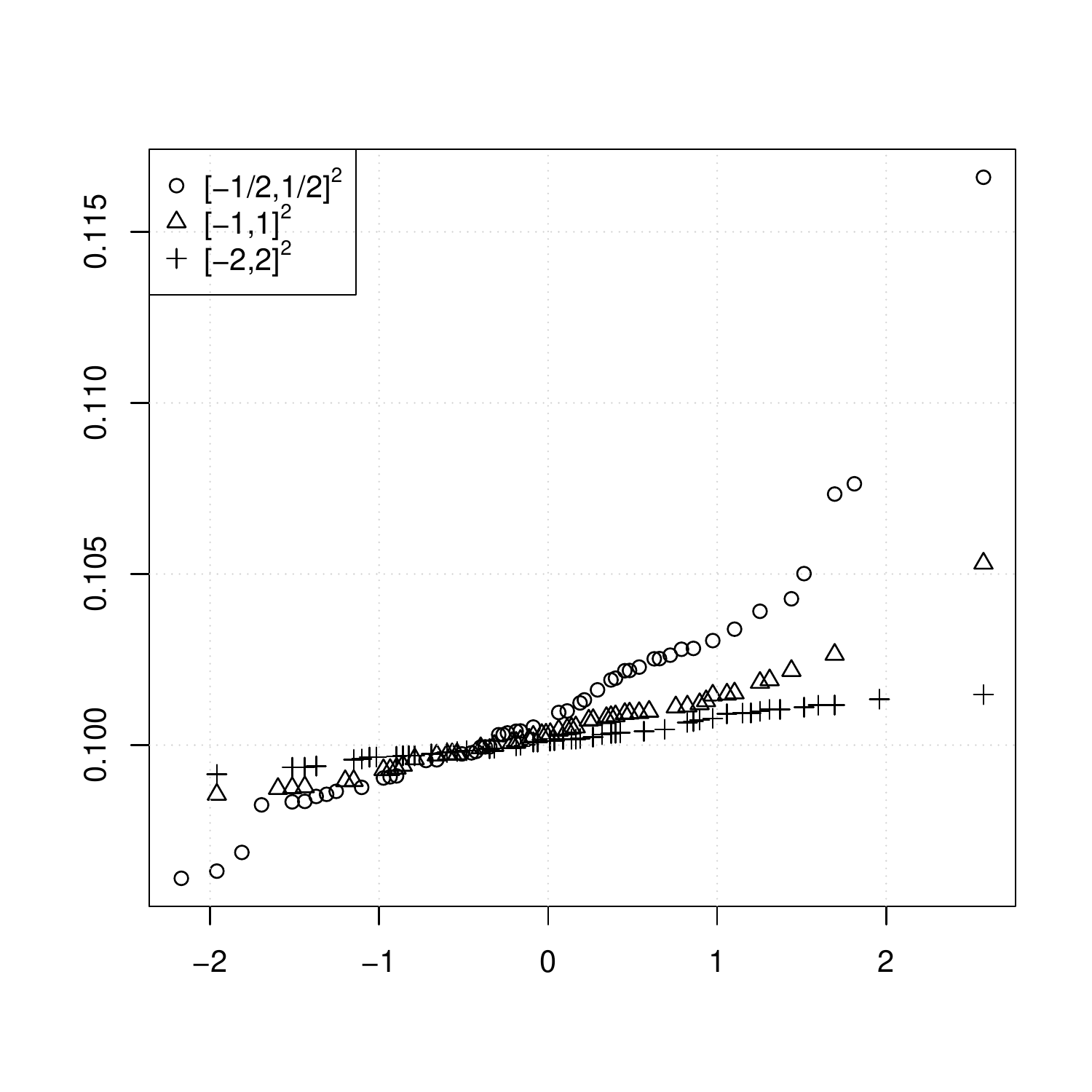} &\includegraphics[scale=.45]{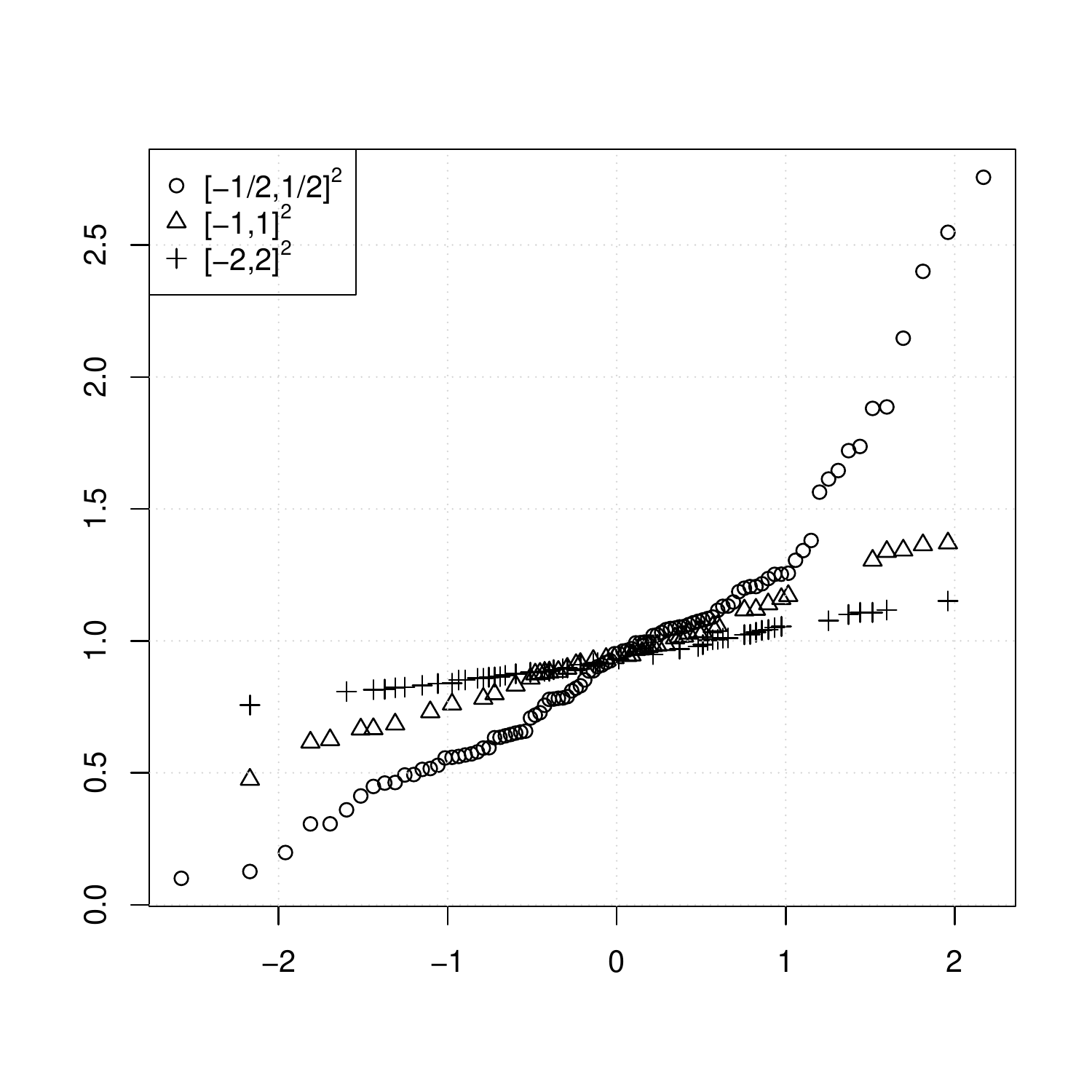} \\
\end{tabular}  
\caption{\label{fig:qqnorm3} Normal QQ-plots for estimates of the parameter $\varepsilon$ (left) and $\sigma$ (right) for the high interaction Lennard-Jones model, i.e. $(\log(\beta),\varepsilon,\sigma)=(\log(100),1,0.1)$. The first row (resp. the second and third) corresponds to estimates obtained with $\alpha_n=R_n=\alpha_{\mathrm{opt}}$ (resp. with $\alpha_n=\alpha_{\mathrm{opt}}, R_n=\infty$ and with $\alpha_n=0,R_n=\infty$). The optimal values  $\alpha_{\mathrm{opt}}$ are given in Table~\ref{tab:rmse}.}
\end{figure}


\appendix

\section{A new central limit theorem}\label{sec:clt} 
When the Gibbs point process has a finite range, the asymptotic normality of the  pseudolikelihood or the logistic regression  estimators are essentially derived from a central limit theorem for conditionally centered random fields, see the references in introduction. This connection comes from the fact that in the finite range case, the score function of the pseudolikelihood (or the logistic regression) is  
conditionally centered, by application of the conditional GNZ formula~\eqref{condGNZ}.  
In the infinite range case, the score functions of the log-pseudolikelihood and the logistic regression are neither centered, nor conditionally centered. 
In  the following theorem,  the conditional centering condition is replaced by condition~(d), which turns out to be sufficient for our application to $s_{\Wna,R_n}(\bX;\theta^\star)$ in Theorem~\ref{thm:cvloi}.  
The other conditions are mainly due to the non-stationary setting induced by the presence of $\alpha_n$ and $R_n$. 
They allow in particular to control the asymptotic behavior of the empirical covariance matrix in \eqref{cvmatrix}. 
 For two square matrices $A,B$ we write $A\geq B$ when $A-B$ is a positive semi-definite matrix.

\begin{theorem} \label{tcl1} For $n\in \NN$ and $j\in \ZZ^d$, let $X_{n,j}$  be a triangular array field in a measurable space $S$. For $n\in \NN$, let $\mathcal I_n\subset \ZZ^d$ and  $\alpha_n\in\mathbb R_+$ such that $|\mathcal I_n|\to \infty$ and $\alpha_n\to \infty$ as $n\to \infty$. Define $S_n=\sum_{j\in \mathcal I_n}Z_{n,j}$ where $Z_{n,j}=f_{n,j}(X_{n,k}, k \in \mathcal K_{n,j})$ with $\mathcal K_{n,j}= \{ k\in \ZZ^d, |k-j|\leq \alpha_n\}$ and where $f_{n,j}:S^{\mathcal K_{n,j}} \to \RR^p$ is a measurable function. We define $\widehat \Sigma_n$ and $\Sigma_n$ by
\[
  \widehat\Sigma_n = \sum_{j \in \mathcal I_n} \sum_{
  \begin{subarray}{c}
  k \in \mathcal I_n \\ |k-j|\leq \alpha_n
  \end{subarray}} Z_{n,j} Z_{n,k}^\top \qquad \mbox{ and } \qquad \Sigma_n = \EE \widehat\Sigma_n. 
\]
We assume that  
\begin{itemize}
\item[(a)] $\EE Z_{n,j} = 0$ and  there exists $q\geq 1$ such that $\ \sup_{n\geq 1}\sup_{j\in \mathcal I_n}\EE \|Z_{n,j}\|^{4q}<\infty$,
\item[(b)] for any sequence $\mathcal J_n\subset \mathcal I_n$  such that $|\mathcal J_n|\to \infty$  as $n\to \infty$, 
\[
 |\mathcal J_n|^{-1} \sum_{
 \begin{subarray}{c}
 j,k\in \mathcal J_n  
 \end{subarray}
 } \big\| 
\EE ( Z_{n,j}Z_{n,k}^\top)
  \big\| =\mathcal O(1).
\]
\end{itemize}
Then if $\alpha_n^{\frac{4q-1}{2q-1}d}=o(|\mathcal I_n|)$ as $n\to \infty$,
\begin{equation}\label{cvmatrix}
  |\mathcal I_n|^{-1} ( \widehat\Sigma_n - \Sigma_n ) \to 0 \quad \mbox{ in } L^{2q}.
\end{equation}
If in addition
\begin{itemize}
\item[(c)]  there exists a positive definite matrix $Q$ such that $|\mathcal I_n|^{-1}\Sigma_n \geq Q$ for $n$ sufficiently large,
\item[(d)] as $n\to \infty$
\[ 
  |\mathcal I_n|^{-1/2} \sum_{j\in \mathcal I_n} \EE \left\| 
\EE \left( Z_{n,j} | \; X_{n,k}, k\neq j \right)
  \right\| \to 0,
\]
\end{itemize}
then
\begin{equation}\label{cvloi}
\Sigma_n^{-1/2}S_n \overset{d}{\to} \mathcal N(0,I_p).\end{equation}
\end{theorem}

Before detailing the proof, let us remark that if assumption~(a) is valid for any $q\geq 1$ then the result remains true if $\alpha_n^{(2+\varepsilon)d}=o(|\mathcal I_n|)$ for any $\varepsilon>0$.

\begin{proof} 
For $m,m^\prime=1,\ldots,p$, let $\Delta_{mm'}= ( \widehat \Sigma_n - \Sigma_n )_{mm^\prime}$. Let $q\geq 1$ be as in assumption~(a), the assertion~\eqref{cvmatrix} will be proved if we prove that $\EE\left( |\mathcal I_n|^{-2q} \Delta_{mm^\prime}^{2q} \right) =o(1)$. We have $\Delta_{mm'}= \sum_{j\in \mathcal I_n} U_{n,j}$ where   
\[
U_{n,j} = \sum_{k \in \mathcal I_{n,j}}
\big\{
Z_{n,j} Z_{n,k}^\top - \EE (Z_{n,j} Z_{n,k}^\top)
\big\}_{mm^\prime}
\]
and $\mathcal I_{n,j}=\{k\in \mathcal I_n: |k-j|\leq \alpha_n\}$.
Let $j_1,\dots,j_{2q} \in \mathcal I_n$ such that $|j_k-j_1|> 4\alpha_n$ for $k=2,\dots,2q$. It is clear that for any $j\in \mathcal I_n$, $U_{n,j}$ depends only on $X_{n,k}$ for $|k-j|\leq 2\alpha_n$. So,
\begin{align*}
  \EE (U_{n,j_1}\dots U_{n,j_{2q}}) &= \EE \left\{ 
\EE \left( U_{n,j_1}\dots U_{n,j_{2q}}  \mid \;  X_{n,k}, |k-j_1|>2\alpha_n \right)
  \right\} \\
  &= \EE \left\{ U_{n,j_2}\dots U_{n,j_{2q}}
 \EE \left( 
 U_{n,j_1} \mid \;  X_{n,k}, |k-j_1|>2\alpha_n 
\right)
  \right\} \\
  &= \EE \left\{U_{n,j_2}\dots U_{n,j_{2q}}
 \EE \left( 
 U_{n,j_1} \right)
  \right\} = 0
\end{align*}
whereby we deduce that 
\[
  \EE ( \Delta_{mm^\prime}^{2q}) = \sum_{\begin{subarray}{c}
    j_1,\dots,j_{2q}\in \mathcal I_n \\ |j_k-j_1|\leq 4\alpha_n, k=2,\dots,2q
  \end{subarray}}
  \EE ( U_{n,j_1} \dots U_{n,j_{2q}}).
\]
Now, by condition (a) and H\"older's inequality, we have for any $j\in \mathcal I_n$
\begin{align*}
\EE U_{n,j}^{2q} &= \sum_{
  k_1,\dots,k_{2q} \in \mathcal I_{n,j}}
\EE \left[ \left\{
Z_{n,j}Z_{n,k_1}^\top -\EE(Z_{n,j}Z_{n,k_1}^\top) \right\}_{mm^\prime}  \dots
\left\{Z_{n,j}Z_{n,k_{2q}}^\top -\EE(Z_{n,j}Z_{n,k_{2q}}^\top) \right\}_{mm^\prime} 
\right]\\
&\leq c  \sum_{ k_1,\dots,k_{2q} \in \mathcal I_{n,j}} \mu_{4q}\\
&\leq c \left(\sup_{j\in \mathcal I_n} \#\big\{ k\in \mathcal I_n, |k-j|\leq \alpha_n\big\} \right)^{2q}  = \mathcal O(\alpha_n^{2qd}).
\end{align*}
From H\"older's inequality, we continue with
\[
  \EE ( \Delta_{mm^\prime}^{2q}) \leq \sum_{\begin{subarray}{c}
    j_1,\dots,j_{2q}\in \mathcal I_n \\ |j_k-j_1|\leq 4\alpha_n, k=2,\dots,2q
  \end{subarray}}
  \EE ( U_{n,j_1}^{2q} )^{1/2q}\EE ( U_{n,j_{2q}}^{2q} )^{1/2q}
= \mathcal O \left\{ \alpha_n^{(4q-1)d} |\mathcal I_n|\right\}
\]
leading to 
\[
  \EE(|\mathcal I_n|^{-2q}\Delta_{mm'}^{2q}) = \mathcal O \left( \frac{\alpha_n^{(4q-1)d}}{|\mathcal I_n|^{2q-1}}\right) = o(1)
\]
by assumption on $\alpha_n$, which completes the proof of \eqref{cvmatrix}. 

We now focus on~\eqref{cvloi} and we let
\[
   \overline S_n=\Sigma_n^{-1/2}S_n, \quad S_{n,j}=\sum_{
 k\in \mathcal I_{n,j}
} Z_{n,k}
\quad \mbox{ and } \quad
\overline S_{n,j}= \Sigma_n^{-1/2} S_{n,j}
 \]  
where we recall the notation $\mathcal I_{n,j}=\mathcal K_{n,j}\cap \mathcal I_n$. According to Stein's method \citep[see][]{bolthausen:82}, in order to show \eqref{cvloi} it suffices to prove that for all $u\in \RR^p$ such that $\|u\|=1$ and for all $\omega\in \RR$
\[
  \EE \big\{ (\ii \omega-u^\top\overline S_n) e^{\ii \omega u^\top \overline S_n} 
  \big\} \to 0
\]
as $n\to \infty$ where $\ii = \sqrt{-1}$. Letting $v=\omega u$, this is equivalent to show that for all $v \in \RR^p
$, $\|\EE(A_n)\| \to 0$ where $A_n= (\ii v-\overline{S}_n) e^{\ii v^\top \overline S_n}$.  We decompose the term $A_n$ in the same spirit as \cite{bolthausen:82} : $A_n = A_{n,1} - A_{n,2} - A_{n,3}$ where
\begin{align*}
A_{n,1} &= \ii e^{\ii v^\top \overline S_n} (I_p- \Sigma_n^{-1/2} \sum_{j\in \mathcal I_n} Z_{n,j}S_{n,j}^\top \Sigma_n^{-1/2}) v \\
&= \ii e^{\ii v^\top \overline S_n} (I_p- \Sigma_n^{-1/2} \widehat \Sigma_n \Sigma_n^{-1/2}) v  \\
A_{n,2}&= e^{\ii v^\top \overline S_n} \Sigma_n^{-1/2} \sum_{j \in \mathcal I_n} Z_{n,j} (1-\ii v^\top \overline S_{n,j} - e^{-\ii v^\top \overline S_{n,j}}) \\
A_{n,3} &= \Sigma_n^{-1/2} \sum_{j \in \mathcal I_n} Z_{n,j} e^{\ii v^\top(\overline S_n -\overline S_{n,j})}
\end{align*}
and prove in the following that $\|\EE A_{n,r}\|\to 0$ for $r=1,2,3$ as $n\to \infty$. 

First, assumption (c) implies that $|\mathcal I_n|^{-1} \Sigma_n$ is a positive definite matrix for $n$ sufficiently large, which is now assumed in the following. By $\ell$ we denote the  constant $p/\lambda_{\min}(Q)$ where $\lambda_{\min}(M)$ stands for the smallest eigenvalue of a positive  definite squared matrix $M$. For $n$ sufficiently large, $\lambda_{\min}(|\mathcal I_n|^{-1}\Sigma_n) \geq \lambda_{\min}(Q)>0$ whereby we deduce
\begin{equation} \label{ell}
  \|\Sigma_n^{-1/2}\| = |\mathcal I_n|^{-1/2 } \tr(|\mathcal I_n| \Sigma_n^{-1} )^{1/2} \leq \ell^{1/2} |\mathcal I_n|^{-1/2}.
\end{equation}
Using this result, Jensen's inequality and the sub-multiplicative property of the Frobenius norm, we get for $q\geq 1$ satisfying (a) and the assumption on $\alpha_n$
\begin{align*}
\|\!\EE A_{n,1}\|^{2q} &\leq \|v\|^{2q} \EE \| I_p -\Sigma_n^{-1/2} \widehat \Sigma_n\Sigma_{n}^{-1/2} \|^{2q}  \\
& \leq \|v\|^{2q} \; \|\Sigma_n^{-1/2}\|^{4q} \EE \| \widehat{\Sigma}_n  - \Sigma_n \|^{2q}\\
&\leq \ell^{2q} \|v\|^{2q}  \,|\mathcal I_n|^{-2q} \EE \| \widehat{\Sigma}_n  - \Sigma_n \|^{2q}
\end{align*}
whereby we deduce that $\|\EE A_{n,1}\| \to 0$ from \eqref{cvmatrix}. 
 
Second, since $|1-e^{-\ii z}-\ii z|\leq z^2/2$ for any $z\in \RR$, we have
\begin{align*}
\|A_{n,2}\| & \leq \frac12 \|\Sigma_n^{-1/2}\| \; \sum_{j \in \mathcal I_n} \|Z_{n,j}\| (v^\top \overline S_{n,j})^2 \\
&\leq  \frac{\|v\|^2}2 \|\Sigma_n^{-1/2}\|^3 
\sum_{j\in \mathcal I_n}  B_{n,j}
\leq   \frac{\ell^{3/2}\|v\|^2}2 |\mathcal I_n|^{-3/2} \sum_{j\in \mathcal I_n} B_{n,j}
\end{align*}
where 
\[
B_{n,j} \, = \,  \|Z_{n,j}\| \;  \| \!\! \! \sum_{ k\in \mathcal I_{n,j} } Z_{n,k}  \|^2
\; = \; \|Z_{n,j}\|  
 \sum_{
  k,k^\prime \in \mathcal I_{n,j}}  Z_{n,k}^\top Z_{n,k^\prime} \, \geq 0.
\]
Let us decompose $B_{n,j}=B_{n,j}^{(1)}+B_{n,j}^{(2)}$ where $B_{n,j}^{(1)}= B_{n,j} \mathbf 1(\|Z_{n,j}\|\leq |\mathcal I_n|^\tau)$ and $B_{n,j}^{(2)}= B_{n,j} \mathbf 1(\|Z_{n,j}\| >  |\mathcal I_n|^{\tau})$ with $1/\tau=2(4p-1)$.  By assumption (b), we have
\begin{align}
\|\EE B_{n,j}^{(1)}\| & \leq  |\mathcal I_n|^{\tau} \sum_{
  k,k^\prime \in \mathcal I_{n,j}     
  }
  |\EE(Z_{n,k}^\top Z_{n,k^\prime})| = \mathcal O ( |\mathcal I_n|^{\tau} |\mathcal I_{n,j}|) =\mathcal O ( |\mathcal I_n|^{\tau} \alpha_n^d). \label{eq:B1}
\end{align}
By assumption (a), using H\"older  and Bienaym\'e-Chebyshev inequalities, we continue with
\begin{align}
\EE B_{n,j}^{(2)} &\leq \sum_{\begin{subarray}{c}
k,k^\prime \in \mathcal I_{n,j}     
  \end{subarray}} \EE \left\{   \|Z_{n,j}\| \,   \|Z_{n,k}\|  \,   \|Z_{n,k'}\|  \, \mathbf 1(\|Z_{n,j} >  |\mathcal I_n|^{\tau}\|)\right\}
\nonumber \\
&\leq \alpha_n^{2d}  \;\mu_{4q}^{\frac3{4q}} \;\PP ( \|Z_{n,j}\| >  |\mathcal I_n|^{\tau})^{\frac{4q-3}{4q}}
\nonumber  \\
&\leq  \alpha_n^{2d}\; \mu_{4q}^{\frac3{4q}} \; \left( 
|\mathcal I_n|^{-4q\tau}\EE \|Z_{n,j}\|^{4q}
\right)^{\frac{4q-3}{4q}} \nonumber \\
&\leq \mu_{4q} \; \alpha_n^{2d} |\mathcal I_n|^{-(4q-3)\tau}.\label{eq:B2}
\end{align}
Combining \eqref{eq:B1}-\eqref{eq:B2}, we deduce that as $n\to \infty$
\begin{align*}
  \|\EE A_{n,2}\|  \leq \EE \|A_{n,2}\|&= \mathcal O \left( \frac{\alpha_n^d}{|\mathcal I_n|^{1/2}}|\mathcal I_n|^{\tau}\right) \; + \; 
  \mathcal O \left(
  \frac{\alpha_n^{2d}}{|\mathcal I_n|^{1/2}} \,|\mathcal I_n|^{-(4q-3)\tau}
  \right) =o\left( 1\right)  
\end{align*}  
by definition of $\tau$ and $\alpha_n$.

 Third, for any $j\in \mathcal I_n$, $\overline S_n-\overline S_{n,j}$ does not depend on $X_{n,j}$. This yields
\begin{align*}
\EE A_{n,3} &= \Sigma_n^{-1/2} \sum_{j \in \mathcal I_n} \EE \big\{  
e^{\ii v^\top (\overline S_n-\overline S_{n,j})} \EE (Z_{n,j} \mid X_{n,k}, k \neq j) 
\big\} 
\end{align*}
whereby we deduce, in view of \eqref{ell}, that
\[
  \|\EE A_{n,3}\| \leq c |\mathcal I_n|^{-1/2} \sum_{j\in\mathcal I_n} \EE \left\|
 \EE \left(
Z_{n,j} \mid X_{n,k}, k \neq j
 \right)
  \right\|
\]
which tends to 0 by assumption~(d).
\end{proof}

\section{Auxiliary results} \label{sec:auxiliary}

We gather in this section several auxiliary results. They are established under the setting, assumptions and notation of Section~\ref{sec:IR}. In particular, we recall that  $\Delta_{j}$ is the cube centered at $j\in \ZZ^d$ with volume 1, $\Delta_{n,j}=\Delta_j \cap (\Wna)$, $\mathcal I_n\subset \ZZ^d$ is the set such that $\Wna=\cup_{j\in \mathcal I_n} \Delta_{n,j}$,   $B_{u,n}=B(u,\alpha_n)$ and for any $\Delta\Subset \RR^d$
\begin{align}
&  s_{\Delta} = s_{\Delta}(\bX;\theta^\star) =   \int_{\Delta} t(u,\bX_{B_{u,n}}) \lth(u,\bX_{B_{u,n}}) \dd u  - \sum_{u\in \bX_{\Delta}} t(u,\bXu \setminus u)  \label{defssprime1} \\
&  s'_{\Delta} = s'_{\Delta}(\bX;\theta^\star) =  \int_{\Delta} t(u,\bX_{B_{u,n}}) \lth(u,\bX) \dd u  -  \sum_{u\in \bX_{\Delta}} t(u,\bX_{B_{u,n}} \setminus u).  \label{defssprime2}
\end{align}

\begin{lemma} \label{lem:XWn} 
Let $j\in \mathcal I_n$ and $u\in \Delta_{n,j}$,  assume  \assg, set  
$\gamma'=\gamma_2-d-\varepsilon$ where $0<\varepsilon<\gamma_2-d$ and define
\begin{align*}
 |t_m|(u,\bx)& = \sum_{v \in \bx} |g_m(v-u)| \\
 G(u,\bx)& = \sum_{v \in \bx } \|v-u\|^{-\gamma_2} \mathbf 1(\|v-u\|\geq r_0) \\
 H(u,\bx) &= \sum_{v \in \bx } \|v-u\|^{-d-\varepsilon} \mathbf 1(\|v-u\|\geq r_0).
 \end{align*} 
Then,  if $\alpha_n\geq r_0$ 
\begin{itemize}
\item[(i)] $|t_m|(u,\bx_{B_{u,n}})|  \leq |t_m|(u,\bx)$
\item[(ii)] $|t_m(u,\bx) - t_m(u,\bx_{B_{u,n}})|  \leq c \min\{G(u,\bx), \alpha_n^{-\gamma'} H(u,\bx)\}$
\item[(iii)] $\forall \theta\in \Theta$, $|\lambda_\theta(u,\bx)-\lambda_\theta(u,\bx_{B_{u,n}})| \leq c\, e^{c\,G(u,\bx)} \lambda_\theta(u,\bx) \min\{G(u,\bx), \alpha_n^{-\gamma'} H(u,\bx)\}$.
\end{itemize}

\end{lemma}

\begin{proof}
The first statement is straightforward from the definition. For the second one, from \assg~and since  $\alpha_n\geq r_0$, 
\begin{align*}
| t_m(u,\bx) - t_m(u,\bx_{B_{u,n}}) |  &=|\sum_{v \in \bx} g_m(v-u)  \mathbf 1(\|v-u\|\geq \alpha_n)|\\
&\leq  c \sum_{v \in \bx } \|v-u\|^{-\gamma_2} \mathbf 1(\|v-u\|\geq \alpha_n),
\end{align*}
which is clearly lower than $c\,  G(u,\bx)$. Pushing one step further, we get
\begin{align*}
|t_m(u,\bx) - t_m(u,\bx_{B_{u,n}})  | &\leq  c\, \alpha_n^{-\gamma'} \sum_{v \in \bx } \|v-u\|^{-d-\varepsilon} \mathbf 1(\|v-u\|\geq \alpha_n) \leq c\,   \alpha_n^{-\gamma'} H(u,\bx),
\end{align*}
which proves (ii). For the third statement, since for all $x$, $|1-e^x|< |x| e^{|x|}$, we have
\begin{align*}
|\lambda_\theta(u,\bx)-\lambda_\theta(u,\bx_{B_{u,n}})| &= \lambda_\theta(u,\bx) 
\Big| 1-  e^{\sum_{v\in \bx_{B_{u,n}^c}}\Phi_\theta(v-u)} \Big| \\
&\leq  \lambda_\theta(u,\bx)  \Big|\sum_{v\in \bx_{B_{u,n}^c}}\Phi_\theta(v-u)\Big| e^{\big|\sum_{v\in \bx_{B_{u,n}^c}}\Phi_\theta(v-u)\big|}.
\end{align*}
The result follows from the same inequalities as before, noting that  
\[ \Big|\sum_{v\in \bx_{B_{u,n}^c}}\Phi_\theta(v-u) \Big|= \Big|\sum_{m=2}^p  \theta_m \sum_{v\in \bx_{B_{u,n}^c}} g_m(v-u) \Big|\leq c \sum_{v \in \bx} |g_m(v-u)|  \mathbf 1(\|v-u\|\geq \alpha_n)\]
where $c=(p-1) \sup_{\theta\in \Theta} \sup_m |\theta_m|<\infty$, since $\Theta$ is bounded. 
\end{proof}

\medskip

\begin{lemma}\label{lem:int} 
  Under the assumption \assg, then for any $\theta\in \Theta$ we have the following statements where $\EE$ denotes the expectation with respect to $P_{\theta^\star}$.
  \begin{itemize}
 \item[(i)] For any $q\geq 0$, $ \EE \{ \lambda_\theta(0,\bX)^{q} \} < \infty$.
  \item[(ii)] Let $f : \RR^d \to \RR$  be  a measurable function such that $|f (u)|\leq c(1+\|u\|)^{-\gamma}$ with $\gamma>d$,  then for any $q\geq 0$ 
  \[\EE\left\{e^{q \left |\sum_{u\in \bX} f(\|u\|)\right |}\right \}<\infty.
  \]
  \item[(iii)] For any $q\geq 0$, $q'>0$ and $\theta\in \Theta$,  $\EE \{ |t_m|(0,\bX)^q \lambda_\theta(0,\bX)^{q'} \} < \infty$.
  \item[(iv)] Let $f_1$ and $f_2$ be two functions as in (ii), then for any $q_1,q_2,q_3\geq 0$ and $q'>0$, 
   \[\EE\left\{  |t_m|(0,\bX)^{q_1} \left |\sum_{u\in \bX} f_1(\|u\|)\right |^{q_2} e^{q_3 \left |\sum_{u\in \bX} f_2(\|u\|)\right |} \lambda_\theta(0,\bX)^{q'} \right \}<\infty.\]

  \end{itemize}
\end{lemma}

\begin{proof}
The first statement is a consequence of Proposition~5.2~(a) in \cite{A-Rue70}.
It relies on the following property, see also \citet[Lemma~2]{mase:95}. If   $\psi : \RR^+ \to \RR^+$ is a decreasing function with $\int_0^\infty \psi(t) t^{d-1}\dd t <\infty$, then for any $q\geq 0$, 
\begin{equation*}\label{intpsi}
\EE(e^{q \sum_{u\in \bX} \psi(\|u\|)})<\infty.
\end{equation*}
The proof of (ii) is an easy consequence of this property. We deduce in particular that all moments of $\sum_{u\in \bX} f(\|u\|)$ exist and are finite. Assuming (iii) is true, then (iv) is a straightforward consequence of the previous properties and H\"older's inequality. 
Let us prove (iii). For any $\varepsilon>0$, using the fact that  for any $q\geq 0,\kappa>0$, $x\mapsto x^q e^{-\kappa x}$ is bounded on $[0,\infty)$, we have
\begin{align*} |t_m|(0,\bx)^q \lambda_\theta(0,\bx)^{q'} & = |t_m|(0,\bx)^q  e^{-{q'} \sum_{k=2}^p \theta_k t_k(0,\bx) }\\
& =   |t_m|(0,\bx)^q e^{-{q'} \varepsilon |\theta_m  t_m(0,\bx) |}  e^{{q'}  \varepsilon |\theta_m  t_m(0,\bx)| -{q'} \sum_{k=2}^p \theta_k t_k(0,\bx) }\\
&\leq c \  e^{- {q'}\sum_{u\in\bx}  \tilde\Phi_\theta(u)},
\end{align*}
where $\tilde\Phi_\theta(u) = \sum_{k=2}^p \theta_k g_k(u) - \varepsilon |\theta_m g_m(u)|$. The proof of (iii) is completed in view of (i) if we show that $\tilde\Phi_\theta$ satisfies \assPhi~for any $\theta$. Write $\tilde\Phi_\theta(u)= \tilde\Phi_1(u) + \tilde\Phi_2(u)$ with 
\[\tilde\Phi_1(u) = \frac{\theta_2} 2 g_2(u) + \sum_{k=3}^p \theta_k g_k(u),\quad    \tilde\Phi_2(u) = \frac{\theta_2} 2 g_2(u)  - \varepsilon |\theta_m g_m(u)|.  \]
From  \assg, we deduce that there exists $r>0$ such that $\|u\|<r$ implies $\tilde\Phi_1(u) > c\|u\|^{-\gamma_1}$. 
Moreover if $m=2$, $\tilde \Phi_2(u) >0$ for all $\|u\|<r_0$, provided $\varepsilon<1/2$. If $m\geq 3$, there exists $r'$ such that $\|u\|<r'$ implies $ |\theta_m g_m(u)|< c_g \|u\|^{-\gamma_1}/(4\varepsilon)$ where $c_g$ is the constant in \assg, yielding $\tilde\Phi_2(u) > (c_g/4) \|u\|^{-\gamma_1}$. In all cases, we obtain that for some $r_1>0$,  $\|u\|<r_1$ implies $\tilde\Phi(u) > c\|u\|^{-\gamma_1}$. On the other hand, it is clear that  if $\|u\|>r_0$ then $|\tilde \Phi(u)|\leq c \|u\|^{-\gamma_2}$ and that $\tilde \Phi_\theta$ is bounded from below, proving that it  satisfies \assPhi.
\end{proof}

\medskip

\begin{lemma}\label{moments} 
Let $j\in \mathcal I_n$ and $s_m=(s_{\Delta_{n,j}})_m$, respectively $s'_m=(s'_{\Delta_{n,j}})_m$, be the $m$-th coordinate of $s_{\Delta_{n,j}}$ given by  \eqref{defssprime1}, respectively $s'_{\Delta_{n,j}}$ given by \eqref{defssprime2}. Under  \assg, if $\alpha_n\geq r_0$ then, for any $q\in\NN$, $\EE (|s_m|^q)<\infty$ and $\EE (|s'_m|^q)<\infty$.
\end{lemma}

\begin{proof}

The proof being similar for $s_m$ and $s'_m$, we only give the details concerning $s_m$. From  \eqref{defssprime1} and the binomial formula, the statement is a consequence of 
\[\EE\left\{ \left|\sum_{u \in \bX_{\Delta_{n,j}}} t_m(u,\bX_{B_{u,n}}\setminus u)\right|^{p_1}  \left|\int_{\Delta_{n,j}} t_m(u,\bX_{B_{u,n}}) \lth(u,\bX_{B_{u,n}})\dd u\right|^{p_2} \right\}<\infty\]
for any $p_1,p_2\in\NN$. Applying the Cauchy-Schwarz's inequality, we consider each term above separately. First, for any $p\in\NN$,  by H\"older's inequality and using Lemma~\ref{lem:XWn} we get
\begin{align*}
\EE\bigg\{ \bigg|\int_{\Delta_{n,j}} & t_m(u,\bX_{B_{u,n}}) \lth(u,\bX_{B_{u,n}})\dd u\bigg|^{p} \bigg\}\\
& \leq c \int_{\Delta_{n,j}} \EE \left\{\left| t_m(u,\bX_{B_{u,n}})\right|^p  \lth(u,\bX_{B_{u,n}})^p \right\}\dd u\\
& \leq c \int_{\Delta_{n,j}} \EE \left[\left| t_m(u,\bX)\right|^p \left\{\lth(u,\bX)^p + |\lth(u,\bX_{B_{u,n}}) -  \lth(u,\bX)|^p\right\}\right]\dd u\\
& \leq c \int_{\Delta_{n,j}} \EE \left[\left| t_m(u,\bX)\right|^p \lth(u,\bX)^p 
\left\{1+  G^p(u,\bx)e^{c\, p\, G(u,\bx)}\right\}\right]\dd u
\end{align*}
 which is finite by Lemma~\ref{lem:int} and the stationarity of $\bX$.

Second, we can prove by induction and successive application of the GNZ formula, see Corollary~3.1 in \cite{DecreusefondFlint14}, that
\begin{multline*}
\EE\left[ \left\{\sum_{u \in \bX_{\Delta_{n,j}}} t_m(u,\bX_{B_{u,n}}\setminus u)\right\}^{p}\right] \\ 
=  \sum_{k=1}^p \sum_{(\mathcal P_1,\dots,\mathcal P_k)\in\mathcal T_p^k} \EE \int_{\Delta_{n,j}^k}  \lth(\{u_1,\dots,u_k\},\bX) \prod_{\ell=1}^k t_m^{|\mathcal P_\ell |} (u_\ell, \bX_{B_{u,n}}\cup \{\mathbf u \setminus u_\ell \}  ) \dd \mathbf u
\end{multline*}
where $\mathcal T_p^k$ is the set of all partitions of $\{1,\dots,p\}$ into $k$ subsets, $|\mathcal P|$ is the cardinality of $\mathcal P$, $\mathbf u = (u_1,\dots,u_k)$ and $\mathbf u \setminus u_\ell = (u_1,\dots,u_{\ell-1},u_{\ell+1},\dots,u_k)$. Since
\[ \lth(\{u_1,\dots,u_k\},\bX) =  \prod_{\ell=1}^k \lth(u_\ell,\bX) \prod_{i=1,i\neq \ell}^k e^{-\Phi_{\theta^\star}(u_i-u_\ell)},\]
we obtain by application of H\"older's inequality, 
\begin{align*}
&\EE\left\{ \left|\sum_{u \in \bX_{\Delta_{n,j}}} t_m(u,\bX_{B_{u,n}}\setminus u)\right|^{p}\right\}\leq \\
& \sum_{k=1}^p \sum_{\mathcal P\in\mathcal T_p^k} \prod_{\ell=1}^k  \EE^{1/k} \int_{\Delta_{n,j}^k}   \lth(u_\ell,\bX)^k |t_m (u_\ell, \bX_{B_{u,n}}\cup\{\mathbf u \setminus u_\ell \})|^{k|\mathcal P|} \prod_{i=1,i\neq \ell}^k e^{-k\Phi_{\theta^\star}(u_i-u_\ell)} \dd \mathbf u. \end{align*}
The proof is completed if we show that all expectations above are finite.  To that end, note that 
\[t_m (u_\ell, \bX_{B_{u,n}}\cup \{\mathbf u \setminus u_\ell \}) = t_m(u_\ell,\bX_{B_{u,n}}) + \sum_{h=1,h\neq \ell}^k g_m(u_h-u_\ell)\]
whereby, denoting $q=k|\mathcal P|$
\begin{align*}
&\EE \int_{\Delta_{n,j}^k}   \lth(u_\ell,\bX)^k |t_m (u_\ell, \bX_{B_{u,n}}\cup\{\mathbf u \setminus u_\ell \})|^{q} \prod_{i=1,i\neq \ell}^k e^{-k\Phi_{\theta^\star}(u_i-u_\ell)} \dd \mathbf u \\
& \leq \sum_{r=0}^{q} {q \choose r}  \int_{\Delta_{n,j}^k}    \left|\sum_{h=1,h\neq \ell}^k g_m(u_h-u_\ell)\right|^r \times \\
&\qquad\qquad\qquad \prod_{i=1,i\neq \ell}^k e^{-k\Phi_{\theta^\star}(u_i-u_\ell)} \EE\left\{|t_m(u_\ell,\bX_{B_{u,n}})|^{q-r}  \lth(u_\ell,\bX)^k\right\}  \dd \mathbf u \\
&  \leq c\sum_{r=0}^{q} \sum_{h=1,h\neq \ell}^k   \int_{\Delta_{n,j}^k}    |g_m(u_h-u_\ell)|^r \times \\
&\qquad\qquad\qquad\prod_{i=1,i\neq \ell}^k e^{-k\Phi_{\theta^\star}(u_i-u_\ell)} \EE\left\{|t_m(u_\ell,\bX_{B_{u,n}})|^{q-r}  \lth(u_\ell,\bX)^k\right\}  \dd \mathbf u.
 \end{align*}
 The  last expectation is finite in view of Lemma~\ref{lem:int}, so the above expression is lower than
 \begin{align*} 
  c \sum_{r=0}^{q} \sum_{h=1,h\neq \ell}^k  & \int_{\Delta_{n,j}^k}    |g_m(u_h-u_\ell)|^r \prod_{i=1,i\neq \ell}^k e^{-k\Phi_{\theta^\star}(u_i-u_\ell)} \dd \mathbf u\\
 & \leq c \left\{  \int_{\RR^d}  e^{-k\Phi_{\theta^\star}(v)} \dd v \right\}^{k-2} \sum_{r=0}^{q}   \  \int_{\RR^d}    |g_m(v)|^r e^{-k\Phi_{\theta^\star}(v)} \dd v,
\end{align*}
which is finite from \assg.
\end{proof}

\medskip

\begin{lemma} \label{lem:cov}
The following properties hold under the assumption \assg.
\begin{itemize}
\item[(i)] For $\Lambda_1,\Lambda_2$ two bounded Borel sets of $\RR^d$
\begin{align*}
& \Cov( s'_{\Lambda_1} ,s'_{\Lambda_2}  ) = \EE \int_{\Lambda_1\cap \Lambda_2} 
 t(u,\bX_{B_{u,n}})t(u,\bX_{B_{u,n}})^\top \lth(u,\bX) \dd u \\
&+  \EE \int_{\Lambda_1}\int_{\Lambda_2} t(u,\bX_{B_{u,n}})t(v,\bX_{B_{v,n}})^\top \{ \lth(u,\bX)\lth(v,\bX)  - \lth(\{u,v\},\bX)\} \dd u\dd v \\
&+ \EE \int_{\Lambda_1}\int_{\Lambda_2}
\Delta_v t(u,\bX_{B_{u,n}}) \{\Delta_u t(v,\bX_{B_{v,n}})\}^\top
\lth(\{u,v\},\bX)\dd u\dd v
\end{align*}
where for any $u,v\in \RR^d$, $\bx \in \Omega$ and any measurable function $f:\RR^d\times \Omega \to \RR^p$, the difference operator $\Delta_v$ is defined by $\Delta_v f(u,\bx)= f(u,\bx\cup v)-f(u,\bx)$. \\

\item[(ii)] Let $(\Delta_n)$ be a sequence of increasing domains such that $\Delta_n\to \RR^d$ as $n\to \infty$, then 
\[
  |\Delta_n|^{-1} \Var (s'_{\Delta_n}) \to \Sigma_\infty    
\]  
where $\Sigma_\infty$ is defined by \eqref{eq:SigmaInfty}.

\item[(iii)]  Let $j,k\in \mathcal I_n$. Then if 
$|k-j|>2 r_0$,
\[
  \left\|\Cov(s'_{\Delta_{n,j}},s'_{\Delta_{n,k}}) \right\|\leq c |k-j|^{-\gamma_2}.
\] 
\item[(iv)] For  any $j,k\in   \mathcal I_n$, if  $\alpha_n\geq r_0$, then 
 \begin{align*}
 \left\| \Cov( s_{\Delta_{n,j}} ,s_{\Delta_{n,k}}) -\Cov( s'_{\Delta_{n,j}} ,s'_{\Delta_{n,k}})\right\|  &\leq 
\frac{c}{\alpha_n^{\gamma^\prime} (1+|k-j|^{\gamma_2})} + \frac{c}{\alpha_n^{2\gamma^\prime}}
\end{align*}
as $n\to\infty$, where we recall that $\gamma'=\gamma_2-d-\varepsilon$ with $0<\varepsilon<\gamma_2-d$.
\end{itemize}
\end{lemma}

\begin{proof}
(i) is a slight extension of \citet[Lemma 3.1]{coeurjolly:rubak:13} where the case $\Lambda_1=\Lambda_2$ was considered. The proof is omitted.

For (ii),  we note that for any $u,v \in \RR^d$, $m\geq 1$ and $\bx \in \Omega$
\begin{equation}\label{relation1}
  \lth(u,\bx)\lth(v,\bx)  - \lth(\{u,v\},\bx) = \lth(u,\bx)\lth(v,\bx) \{ 1-e^{-\Phi_{\theta^\star}(v-u)}\}
\end{equation}
and
\begin{equation}\label{relation2}
  \Delta_v t_m(u,\bx) = t_m(u,\bx \cup v)-t_m(u,\bx) =\left\{
\begin{array}{ll}
0 & \mbox{ if } m=1 \\
  g_m(v) & \mbox{ if } m\geq 2
\end{array}
  \right.
\end{equation}
which leads to $\Delta_u t(v,\bx)=g(v)$. Letting $|t|(u,\bx)=\{|t_m|(u,\bx)\}_{m\geq 1}$ for any $u\in \RR^d$ and $\bx \in \Omega$, we have for any $u,v\in \RR^d$
\[
  \|t(u,\bX_{B_{u,n}}) t(v,\bX_{B_{v,n}})^\top \| \leq \| \, |t|(u,\bX)|t|(v,\bX)^\top \|.
\]
The result is derived using the dominated convergence theorem, the stationarity of $\bX$ and since from Lemma~\ref{lem:int} the random variables $\|\,|t|(0,\bX)|t|(0,\bX)^\top \| \lth(0,\bX)$ and $\| \,|t|(0,\bX)|t|(v,\bX)^\top \| \lth(0,\bX)\lth(v,\bX)$ have expectation uniformly bounded in $v$  while  by \assg
\[
  \int_{\RR^d} |1-e^{-\Phi_{\theta^\star}(v)}|\dd v<\infty 
  \quad \mbox{ and } \quad
  \int_{\RR^d} \|g(v)g(v)^\top\| e^{-\Phi_{\theta^\star}(v)} \dd v<\infty.
\]

To prove (iii), we apply (i) to the disjoint sets  $\Delta_{n,j}$, $\Delta_{n,k}$ and relations \eqref{relation1}-\eqref{relation2} to get
\begin{align}\label{ineqcov}
  & \left\|\Cov( s'_{\Delta_{n,j}} ,s'_{\Delta_{n,k}}  ) \right\| \nonumber \\
  & \leq  \EE \int_{\Delta_{n,j}}\int_{\Delta_{n,k}} \| \,|t|(u,\bX)|t|(v,\bX)^\top\|  \lth(u,\bX)\lth(v,\bX) |1-e^{-\Phi_{\theta^\star}(v-u)}| \dd u\dd v \nonumber\\
& \quad+ \EE \int_{\Delta_{n,j}}\int_{\Delta_{n,k}} \| g(v-u)g(v-u)^\top \| \ \lth(u,\bX)\lth(v,\bX)e^{-\Phi_{\theta^\star}(v-u)} 
\dd u\dd v.
\end{align}
Since $|k-j|>2r_0$, we deduce from  \assg~that for any $(u,v)\in \Delta_{n,j}\times \Delta_{n,k}$ and any $m\geq 2$, $|g_m(v-u)|\leq c |k-j|^{-\gamma_2}$. This leads to 
\[\| g(v-u)g(v-u)^\top \|\leq c |k-j|^{-\gamma_2}.\]  
Similarly since $\Phi_{\theta^\star}=\sum_{m=2}^p \theta_m^\star g_m$,  for any $(u,v)\in \Delta_{n,j}\times \Delta_{n,k}$,  $e^{-\Phi_{\theta^\star}(v-u)}\leq e^{|\Phi_{\theta^\star}(v-u)|}\leq c$ and 
\begin{equation}\label{ineq1}  |1-e^{-\Phi_{\theta^\star}(v-u)}| \leq |\Phi_{\theta^\star}(v-u)|e^{|\Phi_{\theta^\star}(v-u)|}\leq c |k-j|^{-\gamma_2}.\end{equation}
Plugging these inequalities in \eqref{ineqcov} shows (iii), as the remaining terms have finite expectations  from Lemma~\ref{lem:int}.

We now focus on (iv). Let us write
$ s_{\Delta_{n,j}} = s'_{\Delta_{n,j}} + I_j$ where $I_j =\int_{\Delta_{n,j}} \pi_n(u,\bx) \dd u$ and 
\[  \pi_n(u,\bx)= t(u,\bx_{B_{u,n}}) \{ \lth(u,\bx_{B_{u,n}}) - \lth(u,\bx)\}. \]
We have
\begin{align}\label{deccov}
 \Cov( s_{\Delta_{n,j}} , & s_{\Delta_{n,k}})  -\Cov(s'_{\Delta_{n,j}} ,s'_{\Delta_{n,k}}) \nonumber \\
 & =  \EE ( s'_{\Delta_{n,j}}I_k^\top) + \EE ( s'_{\Delta_{n,k}}I_j^\top) + \EE (I_j I_k^\top) - \EE (s_{\Delta_{n,j}})  \EE {(s_{\Delta_{n,k}})}^\top. 
\end{align}
Let us control each term in \eqref{deccov}. From the GNZ formula
\[  \EE ( s'_{\Delta_{n,j}}I_k^\top)  = \EE \int_{\Delta_{n,j}}\int_{\Delta_{n,k}} t(u,\bX) \lth(u,\bX )  \{ \pi_n(v,\bX) -\pi_n(v,\bX\cup u) \}^\top\dd u \dd v.\]
By definition of $\lth$ and $t$ (see \eqref{eq:pigpp} and \eqref{deft}), we have for any $u,v\in \RR^d$ and $\bx \in \Omega$
\begin{align*}
  \pi_n(v,\bx \cup u)  = e^{-\Phi_{\theta^\star}(v-u)}[\pi_n(v,\bx) + g(v-u) \{ \lth(v,\bx_{B_{v,n}} )- \lth(v,\bx )\}],
\end{align*}
whereby
\begin{align}\label{term1}
& \EE ( s'_{\Delta_{n,j}}  I_k^\top)   =  \EE \int_{\Delta_{n,j}}\int_{\Delta_{n,k}}t(u,\bX)\lth(u,\bX)  \big\{ 1-e^{-\Phi_{\theta^\star}(v-u)}\big\}  \pi_n(v,\bX)^\top\dd u \dd v\nonumber\\
 &- \EE \int_{\Delta_{n,j}}\int_{\Delta_{n,k}} t(u,\bX)\lth(u,\bX)  e^{-\Phi_{\theta^\star}(v-u)} g(v-u)^\top \{ \lth(v,\bX_{B_{v,n}} )- \lth(v,\bX )\} \dd u\dd v.
\end{align}
\assg~implies \assPhi~which in turn yields $ |1-e^{-\Phi_{\theta^\star}(v-u)}| \leq 1 + e^{-\Phi_{\theta^\star}(v-u)}\leq c$ since $\Phi_{\theta^\star}$ is  bounded from below. On the other hand, for any $m\geq 2$, denoting $\tilde \Phi_{\theta^\star} = \Phi_{\theta^\star}-\varepsilon |\theta_m^\star g_m|$ for some $\varepsilon>0$, we have $|g_m| e^{-\Phi_{\theta^\star}} = |g_m| e^{-\varepsilon |\theta_m^\star g_m|} e^{-\tilde \Phi_{\theta^\star}}\leq c$ since $x\mapsto x e^{-\kappa x}$ is bounded on $[0,\infty)$ for any $\kappa>0$ and $\tilde \Phi_{\theta^\star}$ satisfies \assPhi~as seen in the proof of Lemma~\ref{lem:int}. This proves that for any $u,v$, $\| e^{-\Phi_{\theta^\star}(v-u)} g(v-u)^\top\|$ is bounded.
Moreover, from \eqref{ineq1}, we know that if  $|k-j|>2r_0$, then $ |1-e^{-\Phi_{\theta^\star}(v-u)}| \leq c |k-j|^{-\gamma_2}$ and similarly  $\|e^{-\Phi_{\theta^\star}(v-u)} g(v-u)^\top\|  \leq c |k-j|^{-\gamma_2}$.  We deduce that for any $u\in\Delta_{n,j}$, any  $v\in \Delta_{n,k}$ and any $j,k$, $ |1-e^{-\Phi_{\theta^\star}(v-u)}| \leq c (1+|k-j|)^{-\gamma_2}$ and $\|e^{-\Phi_{\theta^\star}(v-u)} g(v-u)^\top\|  \leq c (1+|k-j|)^{-\gamma_2}$.
Plugging these inequalities in \eqref{term1} and applying Lemmas~\ref{lem:XWn}-\ref{lem:int} to the remaining terms shows that for any $j,k$
\begin{equation}\label{ineq2}
\| \EE ( s'_{\Delta_{n,j}}  I_k^\top)  \| \leq  \frac{c}{\alpha_n^{\gamma^\prime} (1+|k-j|^{\gamma_2})}.
\end{equation}
The same inequality obviously holds for $\| \EE ( s'_{\Delta_{n,k}}  I_j^\top)  \|$.  For the two last terms in the right hand side of \eqref{deccov}, namely
\[  \EE (I_j I_k^\top) = \EE \int_{\Delta_{n,j}}\int_{\Delta_{n,k}} \pi_n(u,\bX)  \pi_n(v,\bX)^\top  \dd u \dd v\]
and, after application of the GNZ formula,
\[ \EE (s_{\Delta_{n,j}}) \EE (s_{\Delta_{n,k}})^\top =  \int_{\Delta_{n,j}}\int_{\Delta_{n,k}} \EE \pi_n(u,\bX) \EE \pi_n(v,\bX)^\top  \dd u \dd v,\]
we deduce from Lemmas~\ref{lem:XWn}-\ref{lem:int} that their norm is bounded by $\alpha_n^{-2\gamma'}$ for any $j,k$, up to a positive  constant. The latter and \eqref{ineq2} prove (iv). 
\end{proof}

\section*{Acknowledgements}
The authors would like to thank sincerely the referee for providing many interesting remarks and relevant suggestions. Her/his work allowed us to have a new look on a previous version of this paper and allowed us to improve our results in several directions. The research of J.-F. Coeurjolly is partially funded by Persyval-lab EA Oculo-Nimbus.

\bibliographystyle{plainnat} 
\bibliography{clt} 

\end{document}